\documentclass[12pt]{article}
\usepackage[reqno]{amsmath}
\allowdisplaybreaks
\usepackage{hyperref}
\usepackage{amssymb, amsthm, enumerate}
\usepackage[capitalise]{cleveref}

\usepackage{graphicx}
\usepackage{tikz}
\usetikzlibrary{decorations.markings}
\usetikzlibrary{shapes.geometric}
\usetikzlibrary{arrows,arrows.meta}
\usepackage{array}
\usepackage{float}
\usepackage{wrapfig}
\usepackage{fullpage,url}
\usepackage{nicefrac}
\usepackage{subcaption}
\usepackage{xcolor}
\usepackage{longtable}
\expandafter\ifx\csname pdfpagebox\endcsname\relax\else
\fi

\newtheoremstyle{slanted}{3pt}{3pt}{\slshape}{}{\bfseries}{.}{ }{}

{\theoremstyle{slanted}

\newtheorem{theorem}{Theorem}[section]
\newtheorem{lemma}[theorem]{Lemma}

\newtheorem{conjecture}[theorem]{Conjecture}
\newtheorem{openprob}[theorem]{Open Problem} 
\crefname{openprob}{Open Problem}{Open Problems}
\Crefname{openprob}{Open Problem}{Open Problems}}
{ \theoremstyle{definition}
\newtheorem{eg}[theorem]{Example}

}

\newcommand\cA{{\mathcal A}}

\newcommand\re{{\mathbb R}}%reals

\DeclareMathOperator\sument{sum}

\DeclareMathOperator{\tr}{tr}

\title{On cores of distance-regular graphs}

\author{
  Annemarie Geertsema\thanks{Korteweg-de Vries Institute for Mathematics, University of Amsterdam, Amsterdam, The Netherlands}
  \and
  Chris Godsil\thanks{Department of Combinatorics \& Optimization, University of Waterloo, Waterloo, Ontario, Canada. (\texttt{cgodsil@uwaterloo.ca})}
  \and
  Krystal Guo\thanks{Korteweg-de Vries Institute for Mathematics, University of Amsterdam, Amsterdam, The Netherlands. (\texttt{k.guo@uva.nl})}
}

\begin{document}
\maketitle

\begin{abstract}
    We look at the question of which distance-regular graphs are \emph{core-complete}, meaning they are isomorphic to their own core or have a complete core. We build on Roberson's homomorphism matrix approach by which method he proved the Cameron-Kazanidis conjecture that strongly regular graphs are core-complete. We develop the theory of the homomorphism matrix for distance-regular graphs of diameter $d$. 
    
    We derive necessary conditions on the cosines of a distance-regular graph for it to admit an endomorphism into a subgraph of smaller diameter $e<d$. As a consequence of these conditions, we show that if $X$ is a primitive distance-regular graph where the subgraph induced by the set of vertices furthest away from a vertex $v$ is connected, any retraction of $X$ onto a diameter-$d$ subgraph must be an automorphism, which recovers Roberson's result for strongly regular graphs as a special case for diameter $2$. 
    
    We illustrate the application of our necessary conditions through computational results. We find that no antipodal, non-bipartite distance-regular graphs of diameter 3, with degree at most $50$ admits an endomorphism to a diameter 2 subgraph.  We also give many examples of intersection arrays of primitive distance-regular graphs of diameter $3$ which are core-complete. Our methods include standard tools from the theory of association schemes, particularly the  spectral idempotents. 

    \vspace{5pt}

    % 05E30 (Association schemes, strongly regular graphs) – Primary
    % 05C15 (Coloring of graphs and hypergraphs) – Secondary
    % 05C50 (Graphs and linear algebra) 
\end{abstract}
    
%%%%%%%%%%%%%%%%%%%%%%%%%%%%%%%%%%%%%%%%%%%%%%%%%%%%%%%%%%%%%%%%%%%%%%%%%%%%%%%%%%%%%
%Intro
%%%%%%%%%%%%%%%%%%%%%%%%%%%%%%%%%%%%%%%%%%%%%%%%%%%%%%%%%%%%%%%%%%%%%%%%%%%%%%%%%%%%%
\section{Introduction}\label{sec:intro}

The \textsl{core} of a graph $X$ is the graph with the least number of vertices which is homomorphically equivalent to $X$. It is known that the core of $X$, denoted $X^{\bullet}$, is unique up to isomorphism and is an induced subgraph of $X$. It has the same chromatic number and clique number as $X$. The core of a graph also inherits symmetries from the graph itself; the cores of vertex-transitive graphs are vertex-transitive and the cores of arc-transitive graphs are arc-transitive. 
This interplay between symmetry and homomorphism motivates our focus on the cores of highly symmetric graphs, such as distance-regular graphs. 

A graph $X$ is said to be \textsl{core-complete} if $X$ is isomorphic to its core $X^{\bullet}$ or $X^{\bullet}$ is a complete graph. There are many classes of graphs which have been shown to be core-complete, including  rank-3 graphs \cite{CamKaz2008}, distance-transitive graphs \cite{GodRoy2011} and block graphs of Steiner systems and orthogonal arrays \cite{GodRoy2011}. A graph $X$ is a \textsl{pseudocore} if every proper endomorphism of $X$ is a coloring. Another example in this family of results is \cite{HuaHuaZha2015}, which shows the stronger statement, alternating forms graphs are pseudocores. In \cite{CamKaz2008}, the Cameron and Kazanidis conjecture that the class of strongly regular graphs, which can be seen as the combinatorial relaxation of rank-3 graphs, are core-complete. This was proven by Roberson in \cite{Rob2019}; he in fact shows the stronger statement that primitive strongly regular graphs are pseudocores. 

A commonality of these results is that the classes of graphs studied are all distance-regular and thus the following is a natural question:

\begin{openprob}\label{op:drg} Are all distance-regular graphs core-complete? 
\end{openprob}

More generally, we can ask which distance-regular graphs are core-complete or are pseudocores.  Hell and Ne\v{s}et\v{r}il show that it is NP-complete to recognize cores of non-bipartite graphs, even amongst $3$-colourable graphs in \cite{HelNes1992}. 
 Thus it is interesting to find large graph classes where the core has a known, well-behaved form — for example, those that are core-complete.

In this paper, we rigorously develop Roberson's idea of a homomorphism matrix from \cite{Rob2019} for distance-regular graphs and generalize his result about the core-completeness of strongly regular graphs to a statement about primitive distance-regular graphs of diameter $d$ where $\Gamma_d(v)$ induces a connected graph for each vertex $v$. 
For $X,Y$ distance-regular graphs with the same intersection array and an eigenvalue $\theta_j$, we define a 
\textsl{homomorphism matrix} using the spectral idempotents of the distance-regular graphs and information from the homomorphism. As a key consequence of this construction, we show in \cref{lem:help-eq-main-thm} that if $X$ is a connected distance-regular graph of diameter $d$, any endomorphism $\phi$ of $X$ whose image $\phi(X)$ has diameter $e$ must satisfy a linear relation on the values $w(e-1,d)$, $w(e,d)$, and $w(e+1,d)$, which are the \textsl{cosines} of the distance-regular graph, a classical tool from the literature which we will explicate in \cref{sec:cos}. This linear relation arises from partitioning $\Gamma_1(v)$ with respect to a geodetic pair of vertices $u,v$, and leverages both combinatorial structure and spectral information.

Using this necessary condition on the cosines for the existence of  a homomorphism to a subgraph of diameter $e$, we show that  if  $X$ is primitive and the subgraph induced by its furthest layer remains connected, any retraction of $X$ onto a diameter $d$ subgraph is an automorphism (\cref{thm:maintheorem}). This strengthens Roberson's result that strongly regular graphs are core-complete. We also show that if $X$ admits an endomorphism to a strictly smaller-diameter subgraph, namely $e < d$, then there exist non-negative integer parameters $\alpha, \beta, \gamma$ satisfying several constraints on the intersection numbers and eigenvalues in \cref{thm:smallerdiam}. 

We also apply our results to the special case in which $X$ has a \textsl{complete core}. We prove that its smallest eigenvalue satisfies $\theta_d \le -2$ and, if $\theta_d = -2$, then every homomorphism from $X$ to its complete core forces a strict limit on how vertices at distance two can map to the same color or image-vertex. Concretely, the equality case of \cref{thm:e=1} gives us that in any coloring of $X$, each color class can have only $\tfrac{k}{c_2}$ vertices at distance two from a given vertex $x$.

Finally, to illustrate the practical significance of these theorems, we performed a series of computations on distance-regular graphs of diameter $3$ using mostly \texttt{SageMath}\cite{sage}. We give feasible intersection arrays of  primitive distance-regular graphs of diameter $3$ which could have an endomorphism to a diameter  $2$ subgraph. We are also able to give many feasible intersection arrays of primitive distance-regular graphs of diameter $3$ which must be core-complete. We find an absence -- among those with degree $k \leq 50$ -- of any antipodal (and non-bipartite)  distance-regular graphs of diameter $3$ that can have an endomorphism to a diameter $2$ subgraph. These findings show the power of our main theorem and lead us to the following conjecture.

\begin{conjecture}\label{conj:antipodaldiam3}
An antipodal and not bipartite distance-regular graph of diameter $3$ has no endomorphism to a subgraph with diameter $2$.
\end{conjecture}

We now describe the organization of the paper. 
We give necessary background definitions in \cref{sec:prelim}. 
In
\cref{sec:hom-mat}, we develop the homomorphism matrix for homomorphisms between distance-regular graphs. 
We apply this to look at endomorphisms  of distance-regular graphs in \cref{sec:end-drg}.
We then use  this machinery to look at endomorphism from distance-regular graphs to their cores in \cref{sec:cores}.
% \ref{sec:diam3}
We finish with further open problems and connections in \cref{sec:conclusion}.

%%%%%%%%%%%%%%%%%%%%%%%%%%%%%%%%%%%%%%%%%%%%%%%%%%%%%%%%%%%%%%%%%%%%%%%%%%%%%%%%%%%%%
% Preliminary
%%%%%%%%%%%%%%%%%%%%%%%%%%%%%%%%%%%%%%%%%%%%%%%%%%%%%%%%%%%%%%%%%%%%%%%%%%%%%%%%%%%%%
\section{Preliminaries}\label{sec:prelim}

We will give some preliminaries and definitions for graph homomorphisms, distance-regular graphs and the associated cosine sequences, while also establishing the notation for this paper.  

\subsection{Graph homomorphisms}

We begin with some preliminaries about graph homomorphisms.  We follow the notation and definitions from \cite[{\S}6]{GR}, to which we defer to further background.
For graphs $X,Y$ a map $\phi:V(X) \rightarrow V(Y)$ is a \textsl{(graph) homomorphism} if adjacent vertices of $X$ are mapped to adjacent vertices of $Y$.  An \textsl{endomorphism} is a homomorphism from $X$ to $X$. A graph $X$ is a \textsl{core} if every endomorphism of $X$ is an automorphism. A subgraph $Y$ of $X$ is said to be a \textsl{core of $X$} if $Y$ is a core and there is a homomorphism from $X$ to $Y$. We see that $Y$ must then be an induced subgraph. Every graph $X$ has a unique core, up to isomorphism, which is denoted by $X^{\bullet}$. We will consider $X^{\bullet}$ to be an induced subgraph of $X$.

A \textsl{retraction} is a homomorphism $f$ from $X$ to a subgraph $Y$ of $X$ such that the restriction of $f$ to $V(Y)$ is the identity map. In this case, $Y$ is said to be a \textsl{retract} of $X$.

We say that a subgraph $Y$ of $X$ is \textsl{isometric} if $d_Y(u,v) = d_X(u,v)$ for all vertices $u,v$ in $Y$. Every retract of $X$ is an isometric subgraph. It is known that if $\phi$ is a non-trivial retraction of $X$, then there exist two vertices $u$ and $v$ at distance $2$ in $X$ such that $\phi(u) = \phi(v)$. If $\phi$  is an endomorphism of $X$, a pair of vertices $u,v$ is \textsl{geodetic} if $d_X(u,v) = d_{\phi(X)}(\phi(u),\phi(v))$.

\subsection{Distance-regular graphs}

Since we will look at cores and homomorphisms of distance-regular graphs, we will need some preliminaries. We note that the purpose of this section is not to give full definitions and background, but rather to establish notation; we defer to the standard text \cite{BCN} for further background on distance-regular graphs and association schemes. We will use functional notation for the entries of matrices; $M(u,v)$ denotes the $(u,v)$ entry of $M$. The distance between $u,v$ in $X$ is denoted $d_X(u,v)$, where the subscript may be omitted when the context is clear.

A connected graph is said to be \textsl{distance-regular} if there exist numbers $b_i,\ c_i$ for $i \geq 0$ such that for any two vertices $u$ and $v$ at distance $i$, the number of neighbours of $v$ at distance $i-1$ from $u$ is $c_i$ and the number of neighbours of $v$ at distance $i+1$ from $u$ is $b_i$. This definition implies that 
\[ b_0 = c_i + a_i + b_i \]
and that there exists a number $p_{ij}^k$ such that for every pair of vertices $u,v$ at distance $k$, there are $p_{ij}^k$ vertices which are simultaneously at distance $i$ from $u$ and at distance $j$ from $v$, for $i,j,k \in \{0,\ldots, d\}$ where $d$ is the diameter of the graph. Suppose $X$ is a distance-regular graph of diameter $d$. The list of parameters $\{b_0,b_1,...,b_{d-1} ;\ c_1,c_2,...,c_d\}$ is called the \textsl{intersection array} of the graph, and the numbers $p_{ij}^k$ are the \textsl{intersection numbers} for the graph.

We define the \textsl{distance graphs} $X_i$ of $X$ as the graphs with vertex set $V(X)$ and two vertices adjacent if and only if they are at distance $i$ in $X$. Let $A = A(X)$ and define distance matrices $A_i(X) = A(X_i)$ for $i =1,\ldots, d$, and $A_0 =I$, the identity matrix. By definition of a distance-regular graph, the matrices $\{A_0,A_1=A, A_2, \ldots, A_d\}$ satisfy:
\[\sum_{k=0}^d A_k = J, \]
where $J$ denotes the all-ones matrix, and 
\[A_i A_j = \sum_{k = 0}^d p_{ij}^k A_k.\]
In fact, the matrices $\{A_0,...,A_d\}$ form an \textsl{association scheme}. 

The Schur product (or element-wise product) of matrices $M$ and $N$ is defined as:
\[(M \circ N)(a,b) = M(a,b) \cdot N(a,b).\]
Since $A_i \circ A_j = \delta_{ij} A_j$ where $\delta_{ij}$ is the Kronecker delta, we see that $\{A_0,...,A_d\}$ are idempotents with respect to the Schur product and thus generate a matrix algebra, $\cA$, which is closed under Schur product. A set of symmetric and pairwise commuting matrices can be simultaneously diagonalized. A distance-regular graph has exactly $d+1$ distinct eigenvalues $\theta_0,\ldots, \theta_d$ and we let  $E_i$ denote the idempotent (with respect to the usual matrix multiplication) projector onto the $\theta_i$-eigenspace. We refer to $\{E_0,...,E_d\}$ as the \textsl{spectral idempotents} of $X$. Since $\{E_0,...,E_d\}$ also forms a basis for $\cA$, we have two $(d+1)\times (d+1)$ matrices, $P$ and $Q$, which give change-of-basis equations as follows:
\begin{align}
    A_i & = \sum_{j=0}^d P(j,i) E_j , \label{Pij} \\
    E_j & = \frac{1}{n} \sum_{i=0}^d Q(i,j) A_i .\label{Qij} 
    \end{align}
    The matrices $P$ and $Q$ are called the \textsl{eigenmatrices} of the scheme. 

    A distance-regular graph $X$ of diameter $d$ is said to be \textsl{primitive} if the graphs $X_i, i\in\{1,\ldots,d\}$ are all connected, and \textsl{imprimitive} otherwise. If $X_d$ is the disjoint union of cliques of the same size,  the graph  $X$ is said to be \textsl{antipodal} and the cliques in $X_d$ are said to be \textsl{fibres} of $X$. If the valency of  an imprimitive distance-regular graph $X$ is at least $3$, then $X$ is either bipartite or antipodal (see \cite[Theorem 4.2.1]{BCN}).

    \subsection{Sequences of cosines}\label{sec:cos}

We also need to define the cosine sequence of a distance-regular graph. These are implicitly defined in \cite[{\S}4.1]{BCN} and can be found explicitly in \cite[{\S}13]{G93}. Let $X$ be a distance-regular graph of diameter $d$. The spectral idempotent $E_j$ can be written as a linear combination of $A_0,\ldots, A_d$ as in \eqref{Qij} and thus the entry $E_j(x,y)$ depends only on the distance between $x$ and $y$. Further, $E_j$ has a constant diagonal. With this in mind, we say that
the \textsl{$r$-th cosine with respect to $\theta_j$} is given by 
\[
w(r,j) = \frac{E_j(x,y)}{E_j(x,x)}
\]
for $x,y$ vertices at distance $r$ in $X$. Since the spectral idempotents $E_r$ have constant diagonal and $\tr(E_r)=m_r$, where $m_r$ is  the multiplicity of eigenvalue $\theta_r$, we can write 
\[w(r,j) = \frac{n E_j(x,y)}{m_r} ,
\]
where $n$ denotes the number of vertices.
The \textsl{sequence of cosines with respect to $\theta_j$} is 
\[
(w(0,j), w(1,j), \ldots, w(d,j)).
\]
Since $E_j$ has a constant diagonal, we can think of these as the ratios between the distinct entries of $E_j$ and the diagonal entry. 
The number of sign-changes of a sequence $(a_0,\ldots, a_m)$ is the number of indices where $a_ia_{i+1}<0$. We will make use of the following theorem, which we are restating here in our notation.

\begin{theorem}\cite[{\S}13.2 Lemma 2.1]{G93}\label{thm:signchanges} Suppose $X$ is a distance-regular graph of diameter $d$ with distinct eigenvalues $\theta_0> \theta_1 > \cdots > \theta_d$. The cosine sequence with respect to $\theta_j$ has exactly $j$ sign-changes.
    %  and, if $j\geq 1$, the sequence 
    % \[(w(0,j)-w(1,j), w(1,j)- w(2,j), \ldots, w(d-1,j)- w(d,j)). \]
% has exactly $i-1$ sign-changes.
\end{theorem}

In particular, this lemma implies that terms of a sequence of cosines for the least eigenvalue alternate in sign and the cosines for the largest eigenvalue all have the same sign. We will illustrate this with an example.

\begin{eg}
For example, the $5$-cycle $C_5$ is a distance-regular graph of diameter $2$ with intersection array $\{2, 1;1,1\}$. The eigenvalues of $C_5$ are \[\theta_0=2^{(1)} > \theta_1=(\varphi -1)^{(2)} > \theta_2= (-\varphi)^{(2)},\] where the multiplicities are given in superscripts and $\varphi = \nicefrac{1}{2}(1 + \sqrt{5})$ is the golden ratio, satisfying
$\varphi^2 - \varphi -1 =0$.  Let $A$ be the adjacency matrix of $C_5$ and $A_2$ be the distance-2 matrix. Let $J$ be the $5\times 5$ all-ones matrix and $I$ be the identity matrix. Then we may compute that the idempotent projectors are
\[
E_0 = \frac{1}{5}J ,\quad E_1=\frac{1}{5} \left(2 I + (\varphi -1)A - \varphi A_2 \right),\quad E_2= \frac{1}{5} \left(2 I  - \varphi A + (\varphi -1)A_2 \right).
\]
The cosines are summarized in \cref{tab:C5}; \cref{thm:signchanges} tells us that there are no sign-changes in the row for $\theta_0$, one sign-change in the row for $\theta_1$ and two sign-changes in the row for $\theta_2$.
\end{eg}

\begin{table}[htbp]
    \centering
\begin{tabular}{c|ccc}
    & $w(0,j)$ & $w(1,j)$ & $w(2,j)$ \\
    \hline
    $j=0$ & $1$ & $1$ & $1$ \\
    $j=1$ & $1$ & $\frac{\varphi-1}{2}$ & $-\frac{\varphi}{2}$ \\
    $j=2$ & $1$ & $-\frac{\varphi}{2}$ & $\frac{\varphi-1}{2}$ \\
\end{tabular}
\caption{The rows give the sequence of cosine with respect to $\theta_j$ for $C_5$. \label{tab:C5}}
\end{table}

For $\theta_j$, we can consider the map $f(r) = w(r,j)$ and use  the following statement about its injectivity, which we have restated from Lemma 3.1 of \cite[{\S}13.3]{G93}. 

\begin{lemma}\cite[{\S}13.3,Lemma 3.1]{G93} \label{lem:cos-injectivity}
    Suppose $X$ is a distance-regular graph of diameter $d$ and valency $k>2$. Let $\theta_j$ be an eigenvalue of $X$. Then $f(r) = w(r,j)$  is not injective if and only if one of the following holds
    \begin{enumerate}[(a)]
        \item $\theta_j =k$, or;
        \item $\theta_j = -k$ (which holds if and only if $X$ is bipartite), or;
        \item there is an even number of (distinct) eigenvalues of $X$ which are greater than $\theta_j$, and $X$ is antipodal. 
    \end{enumerate}
\end{lemma}

Using the recurrences that come from orthogonal polynomials associated with distance-regular graphs, we can derive the following recurrence relation for the cosines:
\begin{align}
    w(0,j) &= 1, \notag \\
    b_0 w(1,j) &= \theta_j, \notag \\
    b_r w(r+1,j) &= (\theta_j - a_r) w(r,j) - c_r w(r-1,j), & \text{for } r = 1, 2, \ldots, d-1, \notag \\
    (\theta_j - a_d) w(d,j) &= c_d w(d-1,j). \label{eq:cosines-recurrence}
\end{align}

We end with a few common pieces of notation from the literature on distance-regular graphs. A distance-regular graph is regular with valency $b_0$; we will often write $k=b_0$. In any graph $X$, the \textsl{$i$-th neighbourhood of $u$}, denoted $\Gamma_i(u)$, is the set of vertices at distance $i$ from $u$ in $X$. If $X$ is distance-regular of diameter $d$, then 
\[
\Gamma_0(u)  \cup \Gamma_1(u) \cup \cdots \cup \Gamma_d(u)
\]
is the \textsl{distance partition} of $X$ with respect to $u$. For any subset $S\subseteq V(X)$, we denote by $X[S]$ the subgraph of $X$ induced by $S$. 

%%%%%%%%%%%%%%%%%%%%%%%%%%%%%%%%%%%%%%%%%%%%%%%%%%%%%%%%%%%%%%%%%%%%%%%%%%%%%%%%%%%%%
% homomorphism matrix
%%%%%%%%%%%%%%%%%%%%%%%%%%%%%%%%%%%%%%%%%%%%%%%%%%%%%%%%%%%%%%%%%%%%%%%%%%%%%%%%%%%%%
\section{The homomorphism matrix}\label{sec:hom-mat}
Suppose $X$ is a distance-regular graph. Since $X^{\bullet}$ is isomorphic to an induced subgraph of $X$, any homomorphism from $X$ to its core is a homomorphism from $X$ to $X$. In this section, we will look at more general homomorphisms between distance-regular graphs. Note that we do not require our homomorphisms to be surjective and so this will give us a linear algebraic way to analyze homomorphisms between distance-regular graphs and their cores. To this end, we let $X$ and $Y$ will be distance-regular graphs of diameter $d$ with the same intersection array and let $\phi$ be a homomorphism from $X$ to $Y$; we retain these definitions throughout this section.

Let $\theta_0 > \theta_1 > \cdots > \theta_d$ be the distinct eigenvalues of $X$ (and of $Y$), with multiplicities $m_0,\ldots, m_d$ respectively.  We will associate a matrix with $\phi$ with respect to each eigenspace of $X$; these matrices will behave like the idempotent projections onto the eigenspace and they will be equal to the idempotent projections when $\phi$ is an isomorphism. 

 We may write the idempotent matrices of the scheme of $X$ as follows:
\[
E_r = \frac{1}{n} \sum_{i=0}^d Q(i,r) A_i
\]
for $r = 0,\ldots, d$, where $Q$ is the $Q$-matrix of the scheme. Entry-wise, we may rewrite this as:
\[
(E_r)(u,v) = \frac{1}{n} Q(d_X(u,v),r).
\]
For any pair of vertices $u$ and $v$ in $X$,
\[
d_X(u,v) \geq d_Y(\phi(u),\phi(v)),
\]
since the shortest $uv$-path in $X$ is mapped to a $uv$-walk in $Y$ of the same length.

The \textsl{$\theta_r$-homomorphism matrix of $X$ with respect to $\phi$}, denoted $M_r^{\phi}$, is the matrix with entries as follows:
\begin{equation}\label{eq:mdef}
(M_r^{\phi})(u,v) = \frac{1}{n} Q( d_Y(\phi(u), \phi(v)), r).
\end{equation}
We will write $M_r$ for $M_r^{\phi}$ when the context is clear.
We will be comparing $M_r$ with $E_r$; from the definition, we have that $(M_r - E_r)(u,v) = 0$ whenever $d_X(u,v) = d_Y(\phi(u), \phi(v))$. Let $w(0,r),\ldots, w(d,r)$ be the cosine sequence for the $\theta_r$-eigenspace of $X$. Recall from \cref{sec:cos} that $w_(j, r) = \nicefrac{n E_{r}(x,y)}{m_{r}} $ where $n$ is the number of vertices of $X$, $m_{r}$ is the multiplicity of $\theta_r$ and $(x,y)$ are vertices at distance $j$.  In terms of the cosine sequence of $\theta_r$, we can write $M_r$ and $E_r$ as follows:
\begin{equation}\label{eq:mr-cos}
(M_r^{\phi})(u,v) = \frac{m_r}{n} w( d_Y(\phi(u), \phi(v)), r) \text{ and } (E_r)(u,v) = \frac{m_r}{n} w(d_X(u, v), r) .
\end{equation}

First, we will show that $M_r$ is positive semi-definite, like $E_r$.

\begin{lemma}For any $r= 0, \ldots, d$, the matrix $M_r$ is positive semi-definite. \end{lemma}

\begin{proof} Let $F_r$ be the idempotent projector onto the $\theta_r$ eigenspace of $Y$. We have that
\[
(M_r)(u,v) = (F_r)(\phi(u), \phi(v))
\]
and thus $M_r$ is a principal submatrix of $F_r \otimes J_{n\times n}$ and is thus positive semi-definite. \end{proof}

We will need to use the following theorem about real matrices. Note that $\sument(M)$ denotes the sum of all of the entries of matrix $M$.

\begin{theorem}\label{thm:psd-tr}If $M, N$ are real $n\times n$ matrices, then $\tr(MN^T) = \sument (M\circ N)$.\end{theorem}

The following lemma shows that $M_r$ behaves like a projection matrix into the $\theta_r$-eigenspace of $A$, in that $z M_r$ is in the $\theta_r$-eigenspace of $A$ for every vector $z \in \re^n$.

\begin{lemma}\label{lem:mr-eigsp} For $r=0,\ldots, d$, we have that $\tr(M_r^{\phi}(A-\theta_r I )) = 0$ and $M_d^{\phi}(A-\theta_d I ) = 0$.\end{lemma}

\begin{proof}
Consider the matrix $M_r - E_r$. From the definition of $M_r$, we see immediately that $(M_r - E_r)(u,v) = 0$ whenever $\phi$ preserves the distance between $u$ and $v$. In particular, if $u, v$ are adjacent, then
\[ (M_r)(u,v) = (E_r)(u,v) = \frac{1}{n} Q(1,r).\]
Thus we have that
\[
M_r \circ A = E_r \circ A \text{ and } M_r \circ I = E_r \circ I
\]
and thus
\[
M_r \circ (\alpha A + \beta I)= E_r \circ (\alpha A + \beta I)
\]
for any scalars $\alpha,\beta$.
Since $E_r$ is the idempotent projection onto the $\theta_r$-eigenspace, we have that
$E_r(A- \theta_rI) = 0$.
Theorem \ref{thm:psd-tr} gives us that
\[
\begin{split}
\tr(M_r(A-\theta_r I )) &= \sument (M_r \circ (A-\theta_r I )) \\
&= \sument(E_r \circ (A-\theta_r I )) \\
&= \tr(E_r (A-\theta_r I )) \\
&= 0.
\end{split}
\]
We see that $A- \theta_d I \succeq 0$, since all its eigenvalues are non-negative. Since $M_d$ is also positive semi-definite, we have that $M_d(A-\theta_d I )$ is a positive semi-definite matrix whose trace is $0$ and is thus equal to the zero matrix. \end{proof}

Note that if $\phi$ is an isomorphism from $X$ to $Y$, then $d_X(u,v) = d_Y(\phi(u), \phi(v))$ for all $u,v$ and so $M_r = E_r$. Whenever $\phi$ is not an isomorphism, $M_r^{\phi}-E_r$ may give rise to non-trivial eigenvectors in the $\theta_r$ eigenspace of $X$. The next lemma will help us  try to show that such eigenvectors cannot exist in most cases.

\begin{lemma}\label{lem:mr-er} If $M_d$ is the $\theta_d$-homomorphism matrix of $X$ with respect to $\phi$, then
    \[
     \theta_d(M_d - E_d)(u,v) = \sum_{ w \in \Gamma_1(v)} (M_d - E_d)(u,w).
    \]
     \end{lemma}

\begin{proof} Lemma \ref{lem:mr-eigsp} gives that $M_d(A - \theta_d I) = 0$. Since $E_d(A - \theta_d I) = 0$, we see that $(M_d - E_d)(A - \theta_d I) = 0$. We have
\[
((M_d - E_d)(A - \theta_d I))(u,v) = \sum_{w\in V(X)}(M_d - E_d)(u,w)(A - \theta_d I)(w,v). \]
Note that $(M_d - E_d)(u,w) = 0$ whenever $d(u,w) \leq 1$. Since $ (A - \theta_d I)(w,v) =0$ whenever $d(w,v) >1$, we obtain,
\[
\begin{split}
((M_d - E_d)(A - \theta_d I))(u,v)
&= \sum_{\scriptstyle w \in  \{v\} \cup\Gamma_1(v)} (M_d - E_d)(u,w)(A - \theta_d I)(w,v) \\
&= - \theta_d(M_d - E_d)(u,v) + \sum_{ w \in \Gamma_1(v)} (M_d - E_d)(u,w)
\end{split}
\]
and the result follows. \end{proof}

%%%%%%%%%%%%%%%%%%%%%%%%%%%%%%%%%%%%%%%%%%%%%%%%%%%%%%%%%%%%%%%%%%%%%%%%%%%%%%%%%%%%%
% DRGs 
%%%%%%%%%%%%%%%%%%%%%%%%%%%%%%%%%%%%%%%%%%%%%%%%%%%%%%%%%%%%%%%%%%%%%%%%%%%%%%%%%%%%%
\section{Endomorphisms of distance-regular graphs}\label{sec:end-drg}
Now we will consider endomorphisms of distance-regular graphs. Since every endomorphism is a composition of a retraction and an automorphism, we will restrict ourselves to retractions. Let $X$ be a distance-regular graph of diameter $d$ and let $\phi$ be a retraction from $X$ to a subgraph $Y$ of $X$.  Let $e$ be the diameter of $Y$. For $S\subset V(X)$, we will write $\phi(S)$ for the set of images of vertices of $S$. We will suppose that $X$ has at least $2$ vertices. 

Let $u,v$ be vertices of $X$ such that $d_X(u,v) = d_Y(\phi(u), \phi(v)) = e$; such vertices always exist since $Y$ is a retract, and we can take two vertices at maximum distance in $Y$. We will say that such a pair $u,v$ is a \textsl{geodetic pair} of vertices.  Each neighbour $w$ of $v$ is at distance $e-1, e$ or $e+1$ from $u$ and is mapped by $\phi$ to a neighbour of $\phi(v)$ at distance $e$ or $e-1$ from $\phi(u)$. We may partition the neighbours of $v$ in $X$ based on the distance of their images to $u$ in $Y$; that is, we partition  $\Gamma_1(v)$ in $X$  into $C_{a,b}$ for $a \in \{e-1, e, e+1\}$ and $b \in \{e-1, e\}$ as follows:
\[
C_{a,b} = \{ w \in \Gamma_1(v) \ | \ d_X(u,w) = a \text{ and } d_Y(\phi(u), \phi(w)) = b \},
\]
where we let $C_{d+1,b} = \emptyset$ for any $b$, for convenience.  
Note that $C_{e-1,e} = \emptyset$, since the image of $u,w$ cannot be further apart than $u,w$.  These are shown in  \cref{fig:Cab}. We will call the partition \[ \bigcup_{a \in \{e-1,e,e+1\}, b\in \{e-1,e\}} C_{a,b} \] the \textsl{$\phi$-partition of $\Gamma_1(v)$ with respect to $u$}. We will retain these definitions throughout this section, though we will repeat them in the statements of lemmas, since we will apply them in \cref{sec:cores}.

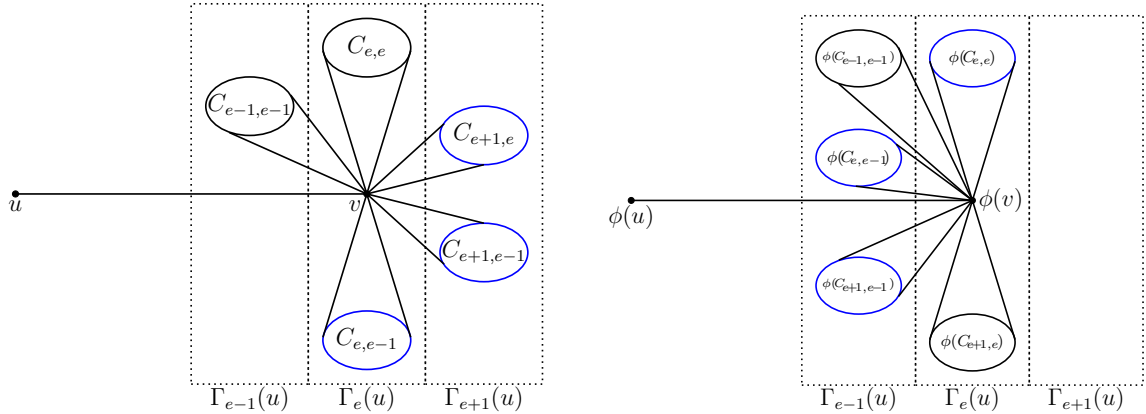
\begin{figure}[htb]
    \centering
    \begin{subfigure}[t]{0.45\textwidth}
        \centering
        \resizebox{\textwidth}{!} {
        \begin{tikzpicture}
            %% vertices
            \draw[fill=black] (0,0) circle (1.5pt);
            \draw[fill=black] (6,0) circle (1.5pt);
            
            %% vertex labels
            \node at (0,-0.2) {$u$};
            \node at (5.8,-0.2) {$v$};
            
            %% neighbourhoods
            \draw[thick] (4,1.5) ellipse (0.75 and 0.5);
            \draw[thick,blue] (6,-2.5) ellipse (0.75 and 0.5);
            \draw[thick] (6,2.5) ellipse (0.75 and 0.5);
            \draw[thick,blue] (8,1) ellipse (0.75 and 0.5);
            \draw[thick,blue] (8,-1) ellipse (0.75 and 0.5);
            
            %% neighbourhood labels
            \node at (4,1.5) {$C_{e-1,e-1}$};
            \node at (6,-2.5) {$C_{e,e-1}$};
            \node at (6,2.5) {$C_{e,e}$};
            \node at (8,1) {$C_{e+1,e}$};
            \node at (8,-1) {$C_{e+1,e-1}$};
            
            %% intersection neighbourhoods
            \draw[thick,dotted] (3,3.25) rectangle ++(2, -6.5);
            \draw[thick,dotted] (5,3.25) rectangle ++(2, -6.5);
            \draw[thick,dotted] (7,3.25) rectangle ++(2, -6.5);
            
            %% intersection neighbourhoods labels
            \node at (4,-3.5) {$\Gamma_{e-1}(u)$};
            \node at (6,-3.5) {$\Gamma_e(u)$};
            \node at (8,-3.5) {$\Gamma_{e+1}(u)$};
            
            %% edges
            \draw[thick] (0,0) -- (6,0);
            
            %Ce-1,e-1
            \draw[thick] (4.7,1.69) -- (6,0);
            \draw[thick] (3.65,1.05) -- (6,0);
            
            %Ce,e
            \draw[thick] (5.25,2.45) -- (6,0);
            \draw[thick] (6.75,2.45) -- (6,0);
            
            %Ce,e-1
            \draw[thick] (5.25,-2.45) -- (6,0);
            \draw[thick] (6.75,-2.45) -- (6,0);
            
            %Ce+1,e
            \draw[thick] (7.32,1.2) -- (6,0);
            \draw[thick] (8,0.5) -- (6,0);
            
            %Ce+1,e-1
            \draw[thick] (7.32,-1.2) -- (6,0);
            \draw[thick] (8,-0.5) -- (6,0);

            \end{tikzpicture} }
        \caption{The sets $C_{a,b}$ for $a \in \{e-1,e,e+1\}$ and $b \in \{e-1,e\}$ in the graph $X$.}
        \label{fig:defcij}
    \end{subfigure} \quad
    \begin{subfigure}[t]{0.45\textwidth}
        \centering
        \resizebox{\textwidth}{!} {
        \begin{tikzpicture}
            %% vertices
            \draw[fill=black] (0,0) circle (1.5pt);
            \draw[fill=black] (6,0) circle (1.5pt);
            
            %% vertex labels
            \node at (0,-0.3) {$\phi(u)$};
            \node at (6.5,0) {$\phi(v)$};
            
            %% neighbourhoods
            \draw[thick] (4,2.5) ellipse (0.75 and 0.5);
            \draw[thick,blue] (4,0.75) ellipse (0.75 and 0.5);
            \draw[thick,blue] (4,-1.5) ellipse (0.75 and 0.5);
            \draw[thick] (6,-2.5) ellipse (0.75 and 0.5);
            \draw[thick,blue] (6,2.5) ellipse (0.75 and 0.5);
            
            %% neighbourhood labels
            \node at (4,2.5) {$\scriptscriptstyle \phi(\!C_{\!e\!-\!1,e\!-\!1})$};
            \node at (4,0.75) {$\scriptstyle\phi(\!C_{\!e,e-1\!}\!)$};
            \node at (4,-1.5) {$\scriptscriptstyle \phi(\!C_{\!e\!+\!  1,e\!-\!1})$};
            \node at (6,2.5) {$\scriptstyle\phi(\!C_{\!e,e}\!)$};
            \node at (6,-2.5) {$\scriptstyle\phi(C_{\!e\!+\! 1,e}\!)$};
            
            %% intersection neighbourhoods
            \draw[thick,dotted] (3,3.25) rectangle ++(2, -6.5);
            \draw[thick,dotted] (5,3.25) rectangle ++(2, -6.5);
            \draw[thick,dotted] (7,3.25) rectangle ++(2, -6.5);
            
            %% intersection neighbourhoods labels
            \node at (4,-3.5) {$\Gamma_{e-1}(u)$};
            \node at (6,-3.5) {$\Gamma_e(u)$};
            \node at (8,-3.5) {$\Gamma_{e+1}(u)$};
            
            %% edges
            \draw[thick] (0,0) -- (6,0);
            
            %Ce-1,e-1
            \draw[thick] (4.7,2.69) -- (6,0);
            \draw[thick] (3.65,2.05) -- (6,0);
            
            %Ce,e-1
            \draw[thick] (4.66,0.99) -- (6,0);
            \draw[thick] (3.95,0.25) -- (6,0);
            
            %Ce+1,e-1
            \draw[thick] (4.7,-1.69) -- (6,0);
            \draw[thick] (3.65,-1.05) -- (6,0);
            
            %Ce+1,e
            \draw[thick] (5.25,-2.45) -- (6,0);
            \draw[thick] (6.75,-2.45) -- (6,0);
            
            %Ce,e
            \draw[thick] (5.25,2.45) -- (6,0);
            \draw[thick] (6.75,2.45) -- (6,0);

            \end{tikzpicture} }
        \caption{The sets $\phi(C_{a,b})$ for $a \in \{e-1,e,e+1\}$ and $b \in \{e-1,e\}$ in graph $Y$, the image of $X$.}
        \label{fig:image-defcij}
    \end{subfigure}
    \caption{Geodetic vertices $u$ and $v$ are at distance $e$ in a distance-regular graph $X$ and the $\phi$-partition of $\Gamma_1(v)$ with respect to $u$.\label{fig:Cab}}
    \end{figure}

    For the following, we will artificially define $w(d+1,d) =0$, for convenience.
    
    \begin{lemma}\label{lem:help-eq-main-thm}
        Suppose $X$ is a connected distance-regular graph with diameter $d$ on more than 2 vertices. Let $\phi$ be an endomorphism of $X$ such that $\phi(X)$ has diameter $e$. Then the following holds for the $\phi$-partition of $\Gamma_1(v)$ with respect to $u$ of any geodetic pair of vertices $u,v$: 
        \begin{align}\label{eq:zero-sum-cij}
            0 = &|C_{e,e-1}| (w(e-1,d) - w(e,d)) + |C_{e+1,e-1}| (w(e-1,d) - w(e+1,d)) \\
            &+ |C_{e+1,e}| (w(e,d) - w(e+1,d)) .\nonumber
        \end{align}
    \end{lemma}
    
    \begin{proof}
    Every endomorphism of $X$ is the composition of an automorphism with a retraction. Thus, it suffices to show the statement for a retraction $\phi$, since the composition of two automorphisms is again an automorphism.
    
    Since $\phi$ is a homomorphism from $X$ to $X$ and we may consider the $\theta_d$-homomorphism matrix $M_d$ with respect to $\phi$. Let $Y$ be the image of $X$ under $\phi$.
    
    Let $u,v$ be two vertices of $X$ such that $d_X(u,v) = d_Y(\phi(u), \phi(v)) = e$. Then we partition, like before, $\Gamma_1(v)$ in $X$ into $C_{a,b}$ for $a \in \{e-1, e, e+1\}$ and $b \in \{e-1, e\}$. For $w \in C_{e-1,e-1} \cup C_{e,e}$, we have that 
    \[
    (M_d - E_d)(u,w) = 0.
    \]
    Since $(M_d - E_d)(u,v) = 0$, \cref{lem:mr-er} gives us that
    \begin{align*}
        0 &= \sum_{w \in \Gamma_1(v)} (M_d-E_d)(u,w) \\
        0 &= \sum_{w \in C_{e,e-1}} (M_d-E_d)(u,w) + \sum_{w \in C_{e+1,e-1}} (M_d-E_d)(u,w) + \sum_{w \in C_{e+1,e}} (M_d-E_d)(u,w) \\
        0 &= \frac{m_d}{n} ( |C_{e,e-1}| (w(e-1,d) - w(e,d)) + |C_{e+1,e-1}| (w(e-1,d) - w(e+1,d)) \\
        &+ |C_{e+1,e}| (w(e,d) - w(e+1,d)) ) \\
        0 &= |C_{e,e-1}| (w(e-1,d) - w(e,d)) + |C_{e+1,e-1}| (w(e-1,d) - w(e+1,d)) \\
        &+ |C_{e+1,e}| (w(e,d) - w(e+1,d)) .\qedhere 
    \end{align*} \end{proof}

    Now we consider how \eqref{eq:zero-sum-cij} can be satisfied. For instance, we can have the following lemma if the two consecutive difference of cosines in  \eqref{eq:zero-sum-cij} have the same sign. For visualization, the sets whose cardinalities are involved in \eqref{eq:zero-sum-cij} are highlighted blue in \cref{fig:Cab}

   We can regroup \eqref{eq:zero-sum-cij} and gather the coefficients of the cosines to obtain:  
    \begin{align*}
    0 = &  (|C_{e,e-1}|+ |C_{e+1,e-1}|  )w(e-1,d) + (|C_{e+1,e}| - |C_{e,e-1}|)w(e,d)   \\ 
        &- ( |C_{e+1,e-1}| + |C_{e+1,e}|  )  w(e+1,d). 
    \end{align*}
    Since $C_{e+1,e-1} \cup C_{e+1,e} = \Gamma_{e+1}(u) \cap \Gamma_1(v)$, we see that 
    \[
     |C_{e+1,e-1}| + |C_{e+1,e}|  = b_e. 
    \]
    Thus \eqref{eq:zero-sum-ws} becomes 
    \begin{equation}\label{eq:zero-sum-ws}
        0 =  (|C_{e,e-1}|+ |C_{e+1,e-1}|  )w(e-1,d) + (|C_{e+1,e}| - |C_{e,e-1}|)w(e,d)   
            - b_e w(e+1,d). 
    \end{equation}
    We can simplify this using properties of distance-regular graphs in the following lemma. 

\begin{lemma}\label{lem:sum-with-intersection-nums}    Suppose $X$ is a connected distance-regular graph with diameter $d$ on more than 2 vertices. Let $\phi$ be an endomorphism of $X$ such that $\phi(X)$  has diameter $e$ and $u,v$ is a geodetic pair of vertices. The $\phi$-partition of $\Gamma_1(v)$ with respect to $u$ satisfies 
\begin{enumerate}[(a)]
    \item $0 =  (|C_{e,e-1}|+ |C_{e+1,e-1}| + c_e )w(e-1,d) + (|C_{e+1,e}| - |C_{e,e-1}| - (\theta_d - a_e))w(e,d) $; and further,
\item $|C_{e+1,e}| - |C_{e,e-1}|  > \theta_d + a_e$. 
\end{enumerate} 
\end{lemma}

\begin{proof} 
    We can combine  with the recurrence for the cosines \eqref{eq:cosines-recurrence} which is 
    \[
    b_e w(e+1,d) = (\theta_d - a_e) w(e,d) - c_e w(e-1, d) 
    \]
    to obtain 
      \begin{align*}
    0 = &  (|C_{e,e-1}|+ |C_{e+1,e-1}|  )w(e-1,d) + (|C_{e+1,e}| - |C_{e,e-1}|)w(e,d)   \\ 
        &- (\theta_d - a_e) w(e,d) + c_e w(e-1, d)  \\ 
    0 = &  (|C_{e,e-1}|+ |C_{e+1,e-1}| + c_e )w(e-1,d) + (|C_{e+1,e}| - |C_{e,e-1}| - (\theta_d - a_e))w(e,d) .  
    \end{align*}  
Since $|C_{e,e-1}|,|C_{e+1,e-1}| \geq 0$ and $c_e \geq 1$, we see that it cannot be true that both terms of this equation are equal to $0$, since $w(e-1,d)\neq 0$. This must be a sum of two terms with opposite sign.  Since the sequence of cosines for $\theta_d$ has $d$ sign-changes by \cref{thm:signchanges}, we see that $w(e-1,d)w(e,d)<0$.  Thus we obtain that the coefficient of $w(e,d)$ must also be strictly positive and part (b) follows.
\end{proof}

Now we will make use of \cref{thm:signchanges}. 

\begin{lemma}\label{lem:diam=e;neq0}
    Suppose $X$ is a connected distance-regular graph with diameter $d$ on more than 2 vertices. Let $\phi$ be an endomorphism of $X$ such that $\phi(X)$ also has diameter $d$. For any  geodetic pair of vertices $u,v$, the $\phi$-partition of $\Gamma_1(v)$ with respect to $u$ satisfies 
    \[ C_{d,d-1} = \{ w \in \Gamma_1(v) \ | \ d_X(u,w) = d \text{ and } d_Y(\phi(u), \phi(w)) = d-1\}
    \]
 is empty. 
\end{lemma}

\begin{proof}
We know that $\phi(X)$ also has diameter $d$, thus $|C_{d+1,d}| = |C_{d+1,d-1}| = 0$. Thus Equation \ref{eq:zero-sum-cij} becomes
\[
0 = |C_{d,d-1}| (w(d-1,d) - w(d,d)) .
\]
\cref{thm:signchanges} gives us that  $w(d-1,d) w(d,d) <0$ and thus neither of them are equal to $0$ and they are not equal. Thus  $|C_{d,d-1}| = 0$.
% \cref{lem:cos-injectivity} says that this happens if and only if $\theta_d = k, \theta_d = -k$ or there is an even number of eigenvalues greater than $\theta_d$ and $X$ is antipodal, where $k$ denotes the degree of $X$. So, $X$ is bipartite or antipodal with even diameter. 
\end{proof}

\begin{lemma}\label{lem:diam>e;Cequal}
    Suppose $X$ is a connected distance-regular graph with diameter $d$ on more than 2 vertices. Let $\phi$ be an endomorphism of $X$ such that $\phi(X)$ has diameter $e < d$. If there exists a geodetic pair of vertices $u,v$ such that $|C_{e+1,e}| = |C_{e,e-1}|$ in $\phi$-partition of $\Gamma_1(v)$ with respect to $u$, then $X$ is either bipartite or antipodal with even diameter.
\end{lemma}

\begin{proof}
Since $|C_{e,e-1}| = |C_{e+1,e}|$, then we obtain that Equation \ref{eq:zero-sum-cij} becomes
\begin{align*}
    0 &= |C_{e+1,e}| (w(e-1,d) - w(e,d) + w(e,d) - w(e+1,d)) \\
    &+ |C_{e+1,e-1}| (w(e-1,d) - w(e+1,d)) \\
    &= |C_{e+1,e}| (w(e-1,d) - w(e+1,d)) + |C_{e+1,e-1}| (w(e-1,d) - w(e+1,d)) \\
    &= (|C_{e+1,e}| + |C_{e+1,e-1}|) (w(e-1,d) - w(e+1,d)) .
\end{align*}

We know that $e<d$, thus the neighbours of $v$ at distance $e+1$ have to be mapped somewhere, so
\[
|C_{e+1,e}| + |C_{e+1,e-1}| = b_e > 0 .
\]
Thus we find by  \eqref{eq:zero-sum-cij} that $w(e-1,d) = w(e+1,d)$. \cref{lem:cos-injectivity} gives us  that $w(*,d)$ is not injective if and only if $\theta_d = k, \theta_d = -k$ or there is an even number of eigenvalues greater than $\theta_d$ and the $X$ is antipodal, where $k$ denotes the degree of $X$. Thus $X$ is bipartite or antipodal with even diameter. \end{proof}

%%%%%%%%%%%%%%%%%%%%%%%%%%%%%%%%%%%%%%%%%%%%%%%%%%%%%%%%%%%%%%%%%%%%%%%%%%%%%%%%%%%%%
% main theorems
%%%%%%%%%%%%%%%%%%%%%%%%%%%%%%%%%%%%%%%%%%%%%%%%%%%%%%%%%%%%%%%%%%%%%%%%%%%%%%%%%%%%%
\section{Cores of distance-regular graphs}\label{sec:cores}

We will use the theory that we have developed for the homomorphism matrix to study endomorphisms of distance-regular graphs. In particular, our goal is to give necessary conditions for the core of a  distance-regular graph to be a proper subgraph. We recall that the core of a graph is always a retract of the graph.

First we consider the case when the core is a complete graph. 

\begin{theorem}\label{thm:e=1} Suppose $X$  is a distance-regular graph of diameter $d>1$ has a complete  core $X^{\bullet}$. Then 
$\theta_d \leq -2$. Further, if $\theta_d = -2$, then for every edge $u,v$ in $X$, at most one neighour $w\neq u$ of $v$ is mapped to the same vertex as $u$, under any homomorphism to $X^{\bullet}$ and thus for any colouring $f$ of $X$ and $x\in V(X)$ such that $f(x) =c$, we have that 
\[
|\{ y \in f^{-1}(c) \mid d(x,y)=2 \} | \leq \frac{k}{c_2}. 
\]
\end{theorem}

\begin{proof} Let $X^{\bullet}$ be the core of $X$. Let $\phi$ be a homomorphism from $X$ to $X^{\bullet}$.  Since $X$ is not itself a complete graph, we may assume that $\phi$ is a non-trivial retraction of $X$. There must exist two vertices of $X$ at distance $2$, which are mapped to the same vertex. Let $u,w$ be such vertices; $u,w$ are vertices of $X$ such that $\phi(u)=\phi(w)$ and $d_X(u,w)=2$. Let $v$ be a vertex on any shortest path from $u$ to $w$. We see that $u,v$ is a geodetic pair and $w \in C_{2,0}$. 

    We will now use part (a) of \cref{lem:sum-with-intersection-nums} with $e=1$. We see that $C_{1,0} =\emptyset$ since common neighbours of $u,v$ cannot be mapped to $u$. We obtain  
    \[
        0 =  (|C_{2,0}| + c_1 )w(0,d) + (|C_{2,1}|  - (\theta_d - a_1))w(1,d) . 
    \]
    We can use the expressions for the cosines from \eqref{eq:cosines-recurrence} and that $c_1=1$ and get 
    \[
        0 =  |C_{2,0}| + 1  + (|C_{2,1}|  - (\theta_d - a_1))\frac{\theta_d}{b_0}. 
    \]
    Since $b_0=k$ and is positive, this hold if and only if 
    \[ \begin{split}
    %    0 &=  k|C_{2,0}| + k  + (|C_{2,1}|  - (\theta_d - a_1)) \theta_d  \\
        0 &=  k|C_{2,0}| + k  + \theta_d |C_{2,1}|  - \theta_d^2 + a_1\theta_d   \\
         &=  (k-\theta_d) |C_{2,0}| + k  + \theta_d (|C_{2,0}| + |C_{2,1}|)  - \theta_d^2 + a_1\theta_d   \\
         &=  (k-\theta_d) |C_{2,0}| + k  + \theta_d b_1 - \theta_d^2 + a_1\theta_d   \\
         &=  (k-\theta_d) |C_{2,0}| + k  + \theta_d (b_1+a_1) - \theta_d^2  \\
    \end{split}
    \]
    where the third equality holds since $ C_{2,0} \cup C_{2,1}$ is the set of neighbours of $v$ at distance $2$ from $u$. Since $w \in C_{2,0}$, we see that $|C_{2,0}| \geq 1$ and we obtain 
    \[ \begin{split}
            0  &=  (k-\theta_d) |C_{2,0}| + k  + \theta_d (b_1+a_1) - \theta_d^2  \\
            &\geq k-\theta_d  + k  + \theta_d (b_1+a_1) - \theta_d^2 \\
            &= 2k-\theta_d   + \theta_d (k - 1) - \theta_d^2 \\
            &= (k - \theta_d)(\theta_d + 2) .
        \end{split}
        \]
    Since $k - \theta_d >0$, we see that $\theta_d \leq -2$. Equality holds if and only if $w$ is the only vertex in $C_{2,1}$.

    We will look more at the equality case, when $\theta_d = -2$.
    We note that the above analysis holds for any choice of $u,v$ such that $u$ is mapped to the same image as a vertex at distance $2$ from $u$.  Now we pick a vertex $u \in X^{\bullet}$. Since $\phi$ is a retraction, we see that $\phi(u)=u$. Suppose $\phi(w)=\phi(w')=u$ for  $w,w'$ both at distance $2$ from $u$. We see that $w,w'$ have no common neighbour in the neighbourhood of $u$, since for any such vertex, the pair $u,v$ would have $w,w' \in C_{2,0}$. Since each of $w,w'$ have $c_2$ neighbours in $\Gamma_1(u)$, we see that there can be at most $k/c_2$ vertices at distance $2$ from $u$ in $\phi^{-1}(u)$. The result follows.
 \end{proof}

 It is well-known that a regular graph with least eigenvalue $\tau>-2$  must be complete or an odd cycle (see, for example, \cite[Cor. 3.12.3]{BCN}). The characterization of distance-regular graphs with least eigenvalue equal to $-2$ is also a classical result; they are either strongly regular (classified by Seidel \cite{Sei1968}) or line graphs (characterized by Mohar and Shawe-Taylor \cite{MohSha1985}). 
The equality characterization implies that if $\omega(X) = \chi(X)$ and $\theta_d = 2$, then, for every colouring of $X$ with $\chi(X)$ colours,  every edge $u,v$ in $X$, at most one neighour $w\neq u$ of $v$ has the same colour as $u$. For example, the line graph of the Tutte $12$-cage is a distance-regular graph of diameter $6$ on 189 vertices with least eigenvalues $-2$. Its chromatic and cliques numbers are both $3$ (verifed using \texttt{SageMath}), and thus its core is $K_3$. For these graphs, \cref{thm:e=1} shows that every colouring of $X$ is a $k/c_2$-improper colouring of the graph whose adjacency matrix is $A + A_2$; for definitions and full context of $d$-improper colourings, see \cite{vdHeuWoo2018}.

Now we turn our attention to endomorphisms of $X$ to subgraphs of the same diameter. We note that any automorphism (including the trivial automorphism) is an example of such an endomorphism. Recall that  $X[S]$ denotes the subgraph of $X$ induced by $S\subseteq V(X)$.

\begin{lemma}\label{lem:helper-lemma-d} Let $X$ be a distance-regular graph of diameter $d$ and let $\phi$ be a retraction endomorphism from $X$ to $Y$, an induced subgraph of $X$.  The following hold:
    \begin{enumerate}[(a)]
        \item $u,w$ is a geodetic pair of  vertices with $d_{Y}(\phi(u),\phi(w)) = d$, for every vertex $w$ in the same component as $v$ of $X[\Gamma_d(u)]$; 
        \item if every component of $X[\Gamma_d(u)]$ contains a vertex mapped to $\Gamma_d(u)$ by $\phi$, then $\phi$ maps $\Gamma_i(u)$ to $\Gamma_i(\phi(u))$ and $u$ is the only vertex mapped to $\phi(u)$. 
    \end{enumerate}
\end{lemma}

\begin{proof} We have from \cref{lem:diam=e;neq0} that for all geodetic pairs of vertices $u,v$ of $X$ with $d_{Y}(\phi(u),\phi(v)) = d$, we have that $C_{d,d-1} = \emptyset$ in the $\phi$-partition of $\Gamma_1(v)$ with respect to $u$. Let $u$ be a vertex in $Y$ such that $u$ has at least one vertex at distance $d$ in $Y$. We then see for any $v$ at distance $d$ from $u$ in $Y$ that any neighbour of $v$ in $\Gamma_d(u)$ in $X$ forms a geodetic pair with $u$, under $\phi$. If $w$ is a neighbour of $v$ in $\Gamma_d(u)$ in $X$, this implies that $d_{Y}(\phi(u),\phi(w)) = d$ and thus, repeating the argument with $u,w$, we obtain that all neighbours of $w$ in $\Gamma_d(u)$ in $X$ form  geodetic pairs with $u$, under $\phi$. Thus, if $x \in \Gamma_d(u)$ in $X$ has a walk to $v$ in the subgraph of $X$ induced by $\Gamma_d(u)$, it follows that $x$ is mapped to a vertex in $\Gamma_d(u)$ in $Y$ by $\phi$. Let $Y_d$ be the subgraph of $X$ induced by $\Gamma_d(u)$. We have shown that if any vertex $v \in \Gamma_d(u)$ in $X$ has $\phi(v) \in \Gamma_d(u)$, then every vertex in the component of  $Y_d$ containing $v$ forms a geodetic pair with $u$, under $\phi$. 

    If every component of $Y_d$ contains a vertex mapped to $\Gamma_d(u)$ by $\phi$, then $\phi$ fixes $\Gamma_d(u)$ setwise. Then, $\phi$ must fix $\Gamma_i(u)$ setwise for $i=0,\ldots,d$, or we would find a shorter walk from $u$ to some vertex of $\Gamma_d(u)$ in $Y$. In particular, we have shown that $u$ is the only vertex mapped to $\phi(u)$, since it is the only vertex at distance $0$ from $u$.  
\end{proof}

In a bipartite distance-regular graph, $a_d=0$ and thus $X[\Gamma_d(X)]$ induces a graph with no edges. For the next theorem, we need $X[\Gamma_d(X)]$ to be connected and that $X$ is not antipodal, thus we can focus on primitive distance-regular graphs.

\begin{theorem}\label{thm:maintheorem} If $X$ is a primitive distance-regular graph and $X[\Gamma_d(X)]$ is connected, then any retraction from $X$ to  a subgraph of  diameter $d$ is an automorphism.
\end{theorem}

\begin{proof} We again let $\phi$ be a retraction  from $X$ to $Y$. 
Let $u,v$ be vertices in $X$ such that  $d_{Y}(\phi(u),\phi(v)) = d$. We can define the $\phi$-partition of $\Gamma_1(v)$ with respect to $u$. With respect to any $u,v$ in $X$ such that $d_{Y}(\phi(u),\phi(v)) = d$, we have that $C_{d,d-1} = \emptyset$. If $X[\Gamma_d(u)]$ is connected, then \cref{lem:helper-lemma-d} gives that $u,w$ is a geodetic pair of  vertices with $d_{Y}(\phi(u),\phi(v)) = d$, for every vertex $w$ in $X[\Gamma_d(u)]$ and thus $\phi$ fixes the distance partition of $u$ setwise and $\phi^{-1}(u) = \{u\}$. But $X[\Gamma_d(w)]$ is also connected, so reversing the roles of $u,w$, we obtain that $\phi$ fixes the distance partition of $w$ setwise and $w$ is the only vertex mapped to $\phi(w)$, for every $w \in \Gamma_d(u)$. 

Now we look at $X_d$, the graph on $V(X)$ where vertices are adjacent if they are at distance $d$ in $X$ (this is the graph where $A(X_d) = A_d$). 
Since  $Y$ has diameter $d$, there exists vertices $u,v$ at distance $d$ in $Y$. Then, since this implies that the distance partition of $u$ is fixed by $\phi$, we see that all neighbours of $u$ in $X_d$ are mapped to neighbours of $\phi(u)$ in $X_d$. Since $X$ is not antipodal, we have that $X_d$ is connected, and iteratively applying this argument to neighbours of $u$ (and neighbours of those vertices) gives us that every pair of vertices at distance $d$ in $X$ is mapped to a pair of vertices at distance $d$ in $Y$. 

 We have shown that if $u$ has a neighbour $v$ in $X_d$ such that $d_{Y}(\phi(u),\phi(v)) = d$, then $u,v$ are fixed by $\phi$ (and are in fibres of size $1$). Thus any vertex $w$ in the same component of $X_d$ as $u,v$ is also fixed by $\phi$. Since $X$ is not antipodal, we have that $X_d$ is connected, and so every fibre of $\phi$ has size $1$, and $\phi$ is an automorphism. 
\end{proof}

We have shown that if $X$ is neither bipartite nor antipodal and $X[\Gamma_d(u)]$ is connected (for all vertices $u$), then it is either a core or the core has diameter strictly smaller than $d$. We note that this is a strengthening of Roberson's result in \cite{Rob2019}. Every primitive strongly regular graph has $X[\Gamma_2(u)]$  connected for all $u$. Thus, this implies an imprimitive strongly regular graph is either a core, or has a core of diameter strictly smaller than $2$, which is to say that it is complete. 

We will now devote the rest of the section to giving a feasibility condition for a primitive or antipodal of odd diameter distance-regular graph to admit an endomorphism to a subgraph of strictly smaller diameter, in terms of the cosines of the distance-regular graph. 

\begin{theorem}\label{thm:smallerdiam}
Suppose $X$ is a distance-regular graph of diameter $d$ which is neither bipartite or antipodal on at least $2$ vertices. If $X$ has an endomorphism $\phi$ such that $\phi(X)$ has diameter $e<d$, then there exists non-negative integers $\alpha,\beta,\gamma$ such that all of the following holds:
\begin{enumerate}[(a)]
    \item $\alpha \leq a_e$ and $\beta + \gamma = b_e$; 
    \item $\gamma - \alpha > \theta_d + a_e$;
    \item $\alpha \neq \gamma$; and,
    \item $\displaystyle     0 = \alpha (w(e-1,d) - w(e,d)) + \beta (w(e-1,d) - w(e+1,d)) + \gamma (w(e,d) - w(e+1,d))$.
\end{enumerate}
\end{theorem}

\begin{proof}
    If $X$ has an endomorphism $\phi$ such that $\phi(X)$ has diameter $e<d$, then there exist a geodetic pair of vertices and we can let $u,v$ be vertices of $X$ such that $d_X(u,v) = d_Y(\phi(u), \phi(v)) = e$ and consider the $\phi$-partition of $\Gamma_1(v)$ with respect to $u$. Then we will see that 
    \[ \alpha = |C_{e,e-1}|,\quad \beta = |C_{e+1,e-1}|, \gamma = |C_{e+1,e}|
    \]
    satisfies (a) -- (d). Part (a) follows since 
    \[ C_{e,e-1} \subset \Gamma_e(u)\cap \Gamma_1(v)  \quad \text{ and }\quad C_{e+1,e-1} \cup C_{e+1,e} \Gamma_{e+1}(u)\cap \Gamma_1(v) 
    \]
     for vertices $u,v$ at distance $e$ in $X$. Part (b) follows from \cref{lem:sum-with-intersection-nums}. If $\alpha = \gamma$, then \cref{lem:diam>e;Cequal} gives us that $X$ is either bipartite or antipodal with even diameter, a contradiction, which gives us (c). Part (d) is exactly \eqref{eq:zero-sum-cij} in \cref{lem:help-eq-main-thm}. 
\end{proof}

The contrapositive of \cref{thm:smallerdiam} gives us that if there does not exist non-negatives integers $\alpha,\beta,\gamma$ which satisfy (a) -- (d), there does not exist an endomorphism to a subgraph of diameter $e$.

%%%%%%%%%%%%%%%%%%%%%%%%%%%%%%%%%%%%%%%%%%%%%%%%%%%%%%%%%%%%%%%%%%%%%%%%%%%%%%%%%%%%%
% diam 3 
%%%%%%%%%%%%%%%%%%%%%%%%%%%%%%%%%%%%%%%%%%%%%%%%%%%%%%%%%%%%%%%%%%%%%%%%%%%%%%%%%%%%%
\section{Diameter three distance-regular graphs}\label{sec:diam3}
In this section, we turn our attention to distance-regular graphs of diameter 3, the smallest open case of Open Problem \ref{op:drg}. Since all bipartite graphs have $K_2$ as their core, we look only at antipodal (but not bipartite) and primitive graphs. To demonstrate the power of \cref{thm:maintheorem} and \cref{thm:smallerdiam}, we use them to find feasible intersection arrays of primitive distance-regular graphs of diameter 3 which are core-complete and  those which could have an endomorphism to a subgraph of diameter 2. We also did  computations for antipodal (but not bipartite) distance-regular graphs of diameter 3, but we did not find any with degree at most $25$ which could admit an endomorphism to a subgraph of diameter $2$.

For the tables of this paper, we generated intersection arrays of distance regular graphs of diameter $3$ which were feasible in that they satisfying the following:
\begin{itemize}
    \item basic integrality and parity checks on $v$, $k_i$'s and $a_i$'s;
    \item the intersection numbers $p_{ij}^k$ are non-negative integers;
    \item all conditions given in \cite[{\S}4.1.D]{BCN} are satisfied;
    \item the multiplicities of the eigenvalues are positive integers;
    \item the Krein numbers $q_{ij}^k$ are non-negative and the Absolute Bound holds (this is  \cite[Proposition 4.1.5]{BCN}); and 
    \item the following theorems about feasibility of intersection array from \cite{BCN}: 
    Cor. 5.1.3,
    Thm. 5.2.5,
    Lem. 5.3.1,
    Thm. 5.4.1,
    Cor. 5.4.2,
    Prop. 5.4.3,
    Prop. 5.5.1*,
    Prop. 5.5.4*
    Lem. 5.5.5,
    Prop. 5.5.6,
    Prop. 5.5.7,
    Prop. 5.6.1,
    Cor. 5.6.2,
    Prop. 5.6.3*,
    Lem. 5.6.4,
    Lem. 5.6.5, 
    Cor. 5.8.2, and 
    Thm. 6.5.1. 
\end{itemize} 
We note that this does not guarantee that the graphs exist and we have not applied all feasbility checks. For the results denoted with ``*'' in the last bullet point from \cite{BCN}, we have used the updated versions from the Additions and Corrections \cite{BCNadd}. We note, in the primitive case, some intersection arrays in \cref{tab:prim3,tab:corecomplete} do not appear in the table of diameter 3 primitive graphs starting page 425 of \cite{BCN}, since we have not applied all known feasibility conditions and the table of \cite{BCN} is restricted to graphs with at most 1024 vertices. The intersection arrays which are not found in the table of \cite{BCN} are denoted with a $-$ before the first column. Our computations were done with \texttt{SageMath}\cite{sage}. 

In \cref{tab:prim3}, we have the intersection arrays of primitive distance-regular graphs $X$ with degree at most $25$ for which there exists at least one triple of non-negative integers $\alpha,\beta,\gamma$ which satisfy the conclusion of \cref{thm:smallerdiam} with $e=2$. Recall that any homomorphism $\phi$ to a subgraph of diameter $2$ will have a geodetic pair of vertices $u,v$ at distance $2$ in both $X$ and $\phi(X)$ and the $\phi$-partition of $\Gamma_1(v)$ with respect to $u$ will give rise to such integers $\alpha,\beta,\gamma$. Thus, these graphs have no homomorphism to a subgraph of diameter $2$. \cref{tab:prim3} covers graphs with degree $k =b_0 \leq 25$; these intersection arrays can also be found in \cite{BCN}.  Since this table is long, to better preserve the readability of the paper, \cref{tab:prim3} can be found in 
\cref{sec:appendixtables}.

        We note that if $\Gamma_3(v)$ for a primitive distance-regular graph of diameter $3$ is connected for all $v$ and it has no numbers $\alpha,\beta,\gamma$ satisfying the conditions in \cref{thm:smallerdiam} for $e=2$, then \cref{thm:maintheorem} gives us that theses graphs must be core-complete. We note that $\Gamma_3(v)$ is an induced subgraph of $X$ which is regular of valency $a_3$. If $a_3 > \theta_1$, then by interlacing, $\Gamma_3(v)$ can only have one eigenvalue equal to $a_3$ and is thus connected. Thus, if $X$ is a primitive distance-regular graph of diameter $3$  has no numbers $\alpha,\beta,\gamma$ satisfying the conditions in \cref{thm:smallerdiam} for $e=2$ and $a_3 > \theta_1$, then $X$ is core-complete. These intersection arrays are given in \cref{tab:corecomplete}. 

    {\footnotesize 
\begin{longtable}{cllr} 
    \caption{Feasible parameters of primitive distance-regular graphs of diameter 3, must be core-complete, with $b_0=k\leq 25$. The first column is the number of vertices, written as the sum of orders of the distance partition of a vertex. The second column gives the eigenvalues of the eigenvalues, with multiplicities shown in superscripts. The third column has the intersection parameters $\{b_0,b_1,b_2; c_1,c_2,c_3\}$; the table is sorted lexicographically by this column. Intersection arrays that are not in \cite{BCN} are denoted with ``$-$'' before the first column. \label{tab:corecomplete}} \\
    \hline
    % Optional headers for the columns on new page:
   & \textbf{Number of vertices} & \textbf{Eigenvalues} & \textbf{Intersection \#s} \\
    \hline
    \endfirsthead
    
    \multicolumn{4}{c}%
    {{\bfseries Table (continued) }}\\[6pt]
    \hline
   & \textbf{Number of vertices} & \textbf{Eigenvalues} & \textbf{Intersection \#s}  \\
    \hline
    \endhead
    
    % -- FIRST ROW --
 & $v = 57 = 1 + 6 + 30 + 20$ & $6^{1}\, 2.618^{18}\, 0.382^{18}\, -3^{20}$ & $\{6,5,2;1,1,3\}$ \\
 & $v = 64 = 1 + 7 + 21 + 35$ & $7^{1}\, 3^{21}\, -1^{35}\, -5^{7}$ & $\{7,6,5;1,2,3\}$ \\
 & $v = 176 = 1 + 7 + 42 + 126$ & $7^{1}\, 3^{66}\, -1^{77}\, -4^{32}$ & $\{7,6,6;1,1,2\}$ \\
 & $v = 135 = 1 + 8 + 56 + 70$ & $8^{1}\, 3^{54}\, -1^{50}\, -4^{30}$ & $\{8,7,5;1,1,4\}$ \\
 & $v = 231 = 1 + 10 + 80 + 140$ & $10^{1}\, 4^{77}\, -1^{98}\, -4^{55}$ & $\{10,8,7;1,1,4\}$ \\
 & $v = 210 = 1 + 11 + 110 + 88$ & $11^{1}\, 4^{55}\, 1^{77}\, -4^{77}$ & $\{11,10,4;1,1,5\}$ \\
 & $v = 175 = 1 + 12 + 72 + 90$ & $12^{1}\, 7^{28}\, 2^{21}\, -2^{125}$ & $\{12,6,5;1,1,4\}$ \\
 & $v = 125 = 1 + 12 + 48 + 64$ & $12^{1}\, 7^{12}\, 2^{48}\, -3^{64}$ & $\{12,8,4;1,2,3\}$ \\
 & $v = 144 = 1 + 13 + 65 + 65$ & $13^{1}\, 5^{39}\, -1^{78}\, -5^{26}$ & $\{13,10,7;1,2,7\}$ \\
 & $v = 216 = 1 + 15 + 75 + 125$ & $15^{1}\, 9^{15}\, 3^{75}\, -3^{125}$ & $\{15,10,5;1,2,3\}$ \\
$-$ & $v = 2057 = 1 + 16 + 240 + 1800$ & $16^{1}\, 5^{680}\, -1^{968}\, -6^{408}$ & $\{16,15,15;1,1,2\}$ \\
 & $v = 324 = 1 + 17 + 136 + 170$ & $17^{1}\, 5^{102}\, -1^{170}\, -7^{51}$ & $\{17,16,10;1,2,8\}$ \\
 & $v = 343 = 1 + 18 + 108 + 216$ & $18^{1}\, 11^{18}\, 4^{108}\, -3^{216}$ & $\{18,12,6;1,2,3\}$ \\
 & $v = 532 = 1 + 18 + 270 + 243$ & $18^{1}\, 5.623^{171}\, -1^{189}\, -4.623^{171}$ & $\{18,15,9;1,1,10\}$ \\
$-$ & $v = 1911 = 1 + 20 + 360 + 1530$ & $20^{1}\, 6^{585}\, -1^{884}\, -6^{441}$ & $\{20,18,17;1,1,4\}$ \\
 & $v = 120 = 1 + 21 + 63 + 35$ & $21^{1}\, 11^{9}\, 3^{35}\, -3^{75}$ & $\{21,12,5;1,4,9\}$ \\
 & $v = 512 = 1 + 21 + 147 + 343$ & $21^{1}\, 13^{21}\, 5^{147}\, -3^{343}$ & $\{21,14,7;1,2,3\}$ \\
 & $v = 330 = 1 + 21 + 168 + 140$ & $21^{1}\, 7.325^{77}\, -1^{175}\, -5.325^{77}$ & $\{21,16,10;1,2,12\}$ \\
$-$ & $v = 650 = 1 + 22 + 396 + 231$ & $22^{1}\, 7^{156}\, 2^{143}\, -4^{350}$ & $\{22,18,7;1,1,12\}$ \\
 & $v = 320 = 1 + 22 + 231 + 66$ & $22^{1}\, 6^{55}\, 2^{154}\, -6^{110}$ & $\{22,21,4;1,2,14\}$ \\
 & $v = 1024 = 1 + 22 + 231 + 770$ & $22^{1}\, 6^{330}\, -2^{616}\, -10^{77}$ & $\{22,21,20;1,2,6\}$ \\
$-$ & $v = 2048 = 1 + 23 + 253 + 1771$ & $23^{1}\, 7^{506}\, -1^{1288}\, -9^{253}$ & $\{23,22,21;1,2,3\}$ \\
 & $v = 165 = 1 + 24 + 84 + 56$ & $24^{1}\, 13^{10}\, 4^{44}\, -3^{110}$ & $\{24,14,6;1,4,9\}$ \\
 & $v = 729 = 1 + 24 + 192 + 512$ & $24^{1}\, 15^{24}\, 6^{192}\, -3^{512}$ & $\{24,16,8;1,2,3\}$ \\
$-$ & $v = 625 = 1 + 24 + 216 + 384$ & $24^{1}\, 9^{120}\, -1^{384}\, -6^{120}$ & $\{24,18,16;1,2,9\}$ \\
$-$ & $v = 8526 = 1 + 25 + 500 + 8000$ & $25^{1}\, 11^{725}\, 4^{2900}\, -4^{4900}$ & $\{25,20,16;1,1,1\}$ \\
    \end{longtable}
    }

We note that there are parameters which appeared in neither \cref{tab:prim3} nor \cref{tab:corecomplete}. For example, the odd graph $O_7$ on $35$ vertices is a primitive distance-regular graph with intersection array $\{4, 3, 3; 1, 1, 2\}$. The subgraph induced by $\Gamma_3(v)$ of any $v$ (since this graph is vertex-transitive) is the disjoint union of three copies of $C_6$. Thus, we cannot apply \cref{thm:maintheorem}. Our computations found that there  were no feasible triples of numbers satisfying the conditions in \cref{thm:smallerdiam}, so its core does not have diameter $2$. We see that $O_7$ is arc-transitive and has $\chi(O_7) = 3$ and $\omega(O_7) = 2$ and thus the coreß must be an arc-transitive graph $Y$ on $5,7$ or $35$ vertices of degree $2$ or $4$ with $\chi(Y) = 3$ and $\omega(Y) = 2$. On $5$ vertices, the only such graph is $C_5$ which has diameter $2$ and thus cannot be the core. On $7$ vertices, the only such graph is $C_7$, which is indeed an induced subgraph of $O_7$. However, we searched all possible homomorphisms using \texttt{SageMath} and see that $O_7$ has no homomorphism to $C_7$, and is thus a core. The argument is ad hoc and reflects the lack of more systematic methods or structural theorems that would apply in this setting. 
 
For antipodal (but not bipartite) distance-regular graphs of diameter 3, we searched up to $b_0\leq 50$ but none of them had any feasible triple $\alpha,\beta,\gamma$ satisfying the conditions in \cref{thm:smallerdiam}. 
These intersection arrays are given in \cref{tab:antipodal3} in 
\cref{sec:appendixtables}. Some of the intersection arrays are known to contain at least one graph; for example, the following all appear in \cref{tab:antipodal3}:
\begin{itemize}
    \item the symplectic $7$-cover of $K_9$ in $63$ vertices with intersection array $\{8,6,1;1,1,8\}$;
    \item $\text{GQ}(2,4)$  minus a spread on $27$ vertices with intersection array $\{8,6,1;1,3,8\}$;
    \item the Coolsaet-Degraer $3$-cover and the symplectic $3$-cover of $K_{14}$, both on $42$ vertices with intersection array $\{13,8,1;1,4,13\}$; and 
\item the symplectic $5$-cover of $K_{12}$ on $60$ vertices with intersection array $\{11,8,1;1,2,11\}$. 
\end{itemize}

We see that the fourth row of  \cref{tab:antipodal3}  contains the intersection array of the Klein graph on $24$ vertices. We can show that this graph is a core. Since this graph is arc-transitive, the degree of the core must divide $7$ and thus, if it had a core on a smaller number of vertices, it would be a $7$-regular graph on $12 $ vertices with chromatic number $4$ and 
clique number $3$. By generating all $7$-regular vertex-transitive graphs on $12$ vertices with \texttt{nauty\_geng}, we determined using \texttt{SageMath} that none of them have chromatic number $4$ and clique number $3$. In comparison, using \cref{thm:smallerdiam} and a much easier computation, we were already able to rule out a core of diameter $2$.  A natural next step would be to find further conditions to eliminate endomorphism to subgraphs of distance-regular graphs with the same diameter.

%%%%%%%%%%%%%%%%%%%%%%%%%%%%%%%%%%%%%%%%%%%%%%%%%%%%%%%%%%%%%%%%%%%%%%%%%%%%%%%%%%%%%
% conclusion 
%%%%%%%%%%%%%%%%%%%%%%%%%%%%%%%%%%%%%%%%%%%%%%%%%%%%%%%%%%%%%%%%%%%%%%%%%%%%%%%%%%%%%
\section{Further directions}\label{sec:conclusion}

In this paper, we introduce a homomorphism matrix for distance-regular graphs and used it to derive structural constraints on endomorphisms to subgraphs of smaller diameter. Our results allow us to show, via a simple computation, that many distance-regular graphs must be core-complete. We also showed how the case of a complete core gives rise to a bound on the smallest eigenvalue, and how equality in that bound enforces a strict local coloring constraint at distance two. 

There are many unanswered questions about homomorphisms of distance-regular graphs, including Open Problem \ref{op:drg} and Conjecture \ref{conj:antipodaldiam3}. \cref{tab:prim3} contains potential counterexamples to the open problem. One can ask whether any of the triples $\alpha,\beta,\gamma$ in \cref{tab:prim3} can be realized as  $|C_{e,e-1}|,|C_{e+1,e-1}|, |C_{e+1,e}|$ for some endomorphism from the graph to a diameter $2$ core. 

We note that it is possible for a core-complete distance-regular graph of diameter $d$ to have an endomorphism to a subgraph of diameter $1<e<d$. The intersection array of Hamming graph $H(3,3)$ on 27 vertices is in the fifth row of \cref{tab:prim3}. One can show that $H(3,3)$ has chromatic number and clique number both equal to $3$  and thus $K_3$ is its core. It has the bowtie graph (two triangles identified at a vertex) as an induced subgraph and it has a homomorphism to the bowtie, which has diameter $2$. This example is due to Roberson \cite{Rob2025} and is also an example of a distance-regular graph which is not a pseudo-core.

Specifically, \cref{tab:prim3} contains the intersection array of point graphs of the generalized hexagons of orders $(s,1)$, for $s=3,4,5,7,8,9, 11, 12$, each with two solutions, $(0,1,s-1)$ and $(1,0,s)$. Let $X$ be the point graph of the generalized hexagon of order $(s,1)$ with $s>2$; the intersection array of $X$ is $\{2s,s,s; 1,1,2\}$ and $a_1= a_2= s-1, a_3 = 2s-s$.  We leave it as an open problem to the reader to show that if $X$ has a retraction $\phi$ to a diameter $2$ subgraph $Y$, then any such retraction must have the property that for any pair of vertices $u,v$ in $Y$ at distance 2 in $Y$, the $\phi$-partition of $\Gamma_1(v)$ with respect to $u$ satisfies the following:
\[ (|C_{e,e-1}|,|C_{e+1,e-1}|, |C_{e+1,e}|) \in \{ (0,1,s-1),(1,0,s)\}.
\]
A follow-up open problem would be to either give a retraction $\phi$ satisfying this condition, or show that it cannot exist.

We note that we have restricted ourselves to looking at the least eigenvalue $\theta_d$ in \cref{sec:cores}, though the analogous statement to \cref{lem:help-eq-main-thm} for any $\theta_i$ can be derived from \cref{lem:mr-er}. We have done this because the sequence of cosines of $\theta_d$ has $d$ sign-changes. It is natural to look at necessary conditions for the existence of endomorphism using other eigenvalues.

It is interesting to ask if matrix-based methods can likely be extended to other highly symmetric families of graphs. One such problem is about the cores of cubelike graphs; a graph is \textsl{cubelike} if it is a Cayley graph of an elementary abelian $2$-group. As the name suggests, the hypercube graphs are examples of cubelike graphs.  In 2008, Ne\v{s}et\v{r}il and \v{S}\'amal \cite{NesSam2008} asked whether the core of a cubelike graph itself cubelike. Despite much partial progress in \cite{ManPivRobRoy2020}, the problem remains open. Though cubelike graphs are not necessarily distance-regular, their adjacency matrices are contained in association schemes. Our homomorphism matrix technique may be extendable to such a setting, for example, to facilitate systematic exploration of candidate non-cubelike cores of cubelike graph. 

Many other questions remain unanswered, including the problems posed in the introduction. For example, analyzing higher-diameter distance-regular graphs, and extending our results to give conditions for classes of distance-regular graphs to be pseudocores would be interesting.

\section*{Acknowledgements} 
We thank Edwin van Dam for helpful discussions about generating feasible intersection arrays of diameter~3 distance-regular graphs.

\newpage 
\appendix
\section{Feasible diameter 3 intersection arrays}\label{sec:appendixtables}
{ \footnotesize
\begin{longtable}{cllrr} % Four columns; adjust as needed
    \caption{Feasible parameters of primitive distance-regular graphs of diameter 3, which have at least one feasible triple $\alpha,\beta,\gamma$ satisfying the conditions in \cref{thm:smallerdiam},  with $b_0=k\leq 25$. The first column is the number of vertices, written as the sum of orders of the distance partition of a vertex. The second column gives the eigenvalues of the eigenvalues, with multiplicities shown in superscripts. The third column has the intersection parameters $\{b_0,b_1,b_2; c_1,c_2,c_3\}$; the table is sorted lexicographically by this column. The last column contains all triples satisfying the conditions in \cref{thm:smallerdiam}.  Intersection arrays that are not in \cite{BCN} are denoted with ``$-$'' before the first column. \label{tab:prim3}} \\
    \hline
    % Optional headers for the columns:
 &   \textbf{Number of vertices} & \textbf{Eigenvalues} & \textbf{Intersection \#s} & \textbf{$\alpha,\beta,\gamma$}\\
    \hline
    \endfirsthead
    
    \multicolumn{5}{c}%
    {{\bfseries Table (continued) }}\\[6pt]
    \hline
  &  \textbf{Number of vertices} & \textbf{Eigenvalues} & \textbf{Intersection \#s} & \textbf{$\alpha,\beta,\gamma$}\\
    \hline
    \endhead
     & $v = 21 = 1 + 4 + 8 + 8$ & $4^{1}\, 2.414^{6}\, -0.414^{6}\, -2^{8}$ & $\{4,2,2;1,1,2\}$ & $(0, 1, 1)$ \\
 & $v = 36 = 1 + 5 + 20 + 10$ & $5^{1}\, 2^{16}\, -1^{10}\, -3^{9}$ & $\{5,4,2;1,1,4\}$ & $(0, 1, 1)$ \\
 & & & & $ (1, 0, 2)$ \\
 & $v = 56 = 1 + 5 + 20 + 30$ & $5^{1}\, 2.414^{20}\, -0.414^{20}\, -3^{15}$ & $\{5,4,3;1,1,2\}$ & $(0, 1, 2)$ \\
 & $v = 52 = 1 + 6 + 18 + 27$ & $6^{1}\, 3.732^{12}\, 0.268^{12}\, -2^{27}$ & $\{6,3,3;1,1,2\}$ & $(0, 1, 2)$ \\
 & & & & $ (1, 0, 3)$ \\
 & $v = 27 = 1 + 6 + 12 + 8$ & $6^{1}\, 3^{6}\, 0^{12}\, -3^{8}$ & $\{6,4,2;1,2,3\}$ & $(0, 1, 1)$ \\
 & & & & $ (1, 0, 2)$ \\
 & $v = 63 = 1 + 6 + 24 + 32$ & $6^{1}\, 3^{21}\, -1^{27}\, -3^{14}$ & $\{6,4,4;1,1,3\}$ & $(0, 2, 2)$ \\
 & $v = 105 = 1 + 8 + 32 + 64$ & $8^{1}\, 5^{20}\, 1^{20}\, -2^{64}$ & $\{8,4,4;1,1,2\}$ & $(0, 1, 3)$ \\
 & & & & $ (1, 0, 4)$ \\
 & $v = 64 = 1 + 9 + 27 + 27$ & $9^{1}\, 5^{9}\, 1^{27}\, -3^{27}$ & $\{9,6,3;1,2,3\}$ & $(0, 1, 2)$ \\
 & & & & $ (1, 0, 3)$ \\
 & $v = 186 = 1 + 10 + 50 + 125$ & $10^{1}\, 6.236^{30}\, 1.764^{30}\, -2^{125}$ & $\{10,5,5;1,1,2\}$ & $(0, 1, 4)$ \\
 & & & & $ (1, 0, 5)$ \\
 & $v = 65 = 1 + 10 + 30 + 24$ & $10^{1}\, 5^{13}\, 0^{26}\, -3^{25}$ & $\{10,6,4;1,2,5\}$ & $(0, 2, 2)$ \\
 & & & & $ (1, 1, 3)$ \\
 & & & & $ (2, 0, 4)$ \\
 & $v = 364 = 1 + 12 + 108 + 243$ & $12^{1}\, 5^{104}\, -1^{168}\, -4^{91}$ & $\{12,9,9;1,1,4\}$ & $(0, 3, 6)$ \\
 & & & & $ (1, 2, 7)$ \\
 & $v = 456 = 1 + 14 + 98 + 343$ & $14^{1}\, 8.646^{56}\, 3.354^{56}\, -2^{343}$ & $\{14,7,7;1,1,2\}$ & $(0, 1, 6)$ \\
 & & & & $ (1, 0, 7)$ \\
$-$ & $v = 255 = 1 + 14 + 112 + 128$ & $14^{1}\, 7^{51}\, -1^{119}\, -3^{84}$ & $\{14,8,8;1,1,7\}$ & $(0, 4, 4)$ \\
 & & & & $ (1, 3, 5)$ \\
 & & & & $ (2, 2, 6)$ \\
 & & & & $ (3, 1, 7)$ \\
 & & & & $ (4, 0, 8)$ \\
 & $v = 135 = 1 + 14 + 56 + 64$ & $14^{1}\, 5^{35}\, -1^{84}\, -7^{15}$ & $\{14,12,8;1,3,7\}$ & $(0, 4, 4)$ \\
 & & & & $ (1, 3, 5)$ \\
 & & & & $ (2, 2, 6)$ \\
 & $v = 855 = 1 + 14 + 168 + 672$ & $14^{1}\, 5^{266}\, -1^{399}\, -5^{189}$ & $\{14,12,12;1,1,3\}$ & $(0, 3, 9)$ \\
 & $v = 160 = 1 + 15 + 90 + 54$ & $15^{1}\, 5^{48}\, -1^{75}\, -5^{36}$ & $\{15,12,6;1,2,10\}$ & $(0, 3, 3)$ \\
 & & & & $ (1, 2, 4)$ \\
 & & & & $ (2, 1, 5)$ \\
 & & & & $ (3, 0, 6)$ \\
 & $v = 506 = 1 + 15 + 210 + 280$ & $15^{1}\, 4^{230}\, -3^{253}\, -8^{22}$ & $\{15,14,12;1,1,9\}$ & $(0, 6, 6)$ \\
 & & & & $ (1, 5, 7)$ \\
 & $v = 657 = 1 + 16 + 128 + 512$ & $16^{1}\, 9.828^{72}\, 4.172^{72}\, -2^{512}$ & $\{16,8,8;1,1,2\}$ & $(0, 1, 7)$ \\
 & & & & $ (1, 0, 8)$ \\
 & $v = 910 = 1 + 18 + 162 + 729$ & $18^{1}\, 11^{90}\, 5^{90}\, -2^{729}$ & $\{18,9,9;1,1,2\}$ & $(0, 1, 8)$ \\
 & & & & $ (1, 0, 9)$ \\
 & $v = 819 = 1 + 18 + 288 + 512$ & $18^{1}\, 5^{324}\, -3^{468}\, -9^{26}$ & $\{18,16,16;1,1,9\}$ & $(0, 8, 8)$ \\
 & $v = 324 = 1 + 19 + 152 + 152$ & $19^{1}\, 7^{57}\, 1^{152}\, -5^{114}$ & $\{19,16,8;1,2,8\}$ & $(0, 2, 6)$ \\
 & & & & $ (1, 1, 7)$ \\
 & & & & $ (2, 0, 8)$ \\
$-$ & $v = 1365 = 1 + 20 + 320 + 1024$ & $20^{1}\, 7^{350}\, -1^{650}\, -5^{364}$ & $\{20,16,16;1,1,5\}$ & $(0, 4, 12)$ \\
 & & & & $ (1, 3, 13)$ \\
 & & & & $ (2, 2, 14)$ \\
 & $v = 792 = 1 + 21 + 420 + 350$ & $21^{1}\, 5^{315}\, -1^{252}\, -6^{224}$ & $\{21,20,10;1,1,12\}$ & $(0, 2, 8)$ \\
 & & & & $ (1, 1, 9)$ \\
 & & & & $ (2, 0, 10)$ \\
 & $v = 512 = 1 + 21 + 210 + 280$ & $21^{1}\, 5^{210}\, -3^{280}\, -11^{21}$ & $\{21,20,16;1,2,12\}$ & $(0, 8, 8)$ \\
 & & & & $ (1, 7, 9)$ \\
 & & & & $ (2, 6, 10)$ \\
$-$ & $v = 1596 = 1 + 22 + 242 + 1331$ & $22^{1}\, 13.317^{132}\, 6.683^{132}\, -2^{1331}$ & $\{22,11,11;1,1,2\}$ & $(0, 1, 10)$ \\
 & & & & $ (1, 0, 11)$ \\
$-$ & $v = 2041 = 1 + 24 + 288 + 1728$ & $24^{1}\, 14.464^{156}\, 7.536^{156}\, -2^{1728}$ & $\{24,12,12;1,1,2\}$ & $(0, 1, 11)$ \\
 & & & & $ (1, 0, 12)$ \\
$-$ & $v = 2457 = 1 + 24 + 384 + 2048$ & $24^{1}\, 11^{324}\, 3^{468}\, -3^{1664}$ & $\{24,16,16;1,1,3\}$ & $(0, 2, 14)$ \\
 & & & & $ (1, 1, 15)$ \\
 & & & & $ (2, 0, 16)$ \\
$-$ & $v = 256 = 1 + 24 + 126 + 105$ & $24^{1}\, 8^{42}\, 0^{168}\, -8^{45}$ & $\{24,21,10;1,4,12\}$ & $(0, 4, 6)$ \\
 & & & & $ (1, 3, 7)$ \\
 & & & & $ (2, 2, 8)$ \\
 & & & & $ (3, 1, 9)$ \\
 & & & & $ (4, 0, 10)$ \\
 & $v = 729 = 1 + 24 + 264 + 440$ & $24^{1}\, 6^{264}\, -3^{440}\, -12^{24}$ & $\{24,22,20;1,2,12\}$ & $(0, 10, 10)$ \\
 & & & & $ (1, 9, 11)$ \\
$-$ & $v = 1176 = 1 + 25 + 400 + 750$ & $25^{1}\, 11^{180}\, 1^{245}\, -3^{750}$ & $\{25,16,15;1,1,8\}$ & $(0, 3, 12)$ \\
 & & & & $ (1, 2, 13)$ \\
 & & & & $ (2, 1, 14)$ \\
 & & & & $ (3, 0, 15)$ \\

    \end{longtable} }

{ \footnotesize
\begin{longtable}{llr} % Four columns; adjust as needed
    \caption{Feasible parameters of antipodal but not bipartite distance-regular graphs of diameter 3, which have no feasible triple $\alpha,\beta,\gamma$ satisfying the conditions in \cref{thm:smallerdiam}, and thus no homomorphism to a subgraph of diameter $2$,  with $b_0=k\leq 50$. The first column is the number of vertices, written as the sum of orders of the distance partition of a vertex. The second column gives the eigenvalues of the eigenvalues, with multiplicities shown in superscripts. The third column has the intersection parameters $\{b_0,b_1,b_2; c_1,c_2,c_3\}$; the table is sorted lexicographically by this column.  \label{tab:antipodal3}} \\
    \hline
    % Optional headers for the columns:
    \textbf{Number of vertices} & \textbf{Eigenvalues} & \textbf{Intersection \#s} \\
    \hline
    \endfirsthead
    
    \multicolumn{3}{c}%
    {{\bfseries Table (continued) }}\\[6pt]
    \hline
    \textbf{Number of vertices} & \textbf{Eigenvalues} & \textbf{Intersection \#s} \\
    \hline
    \endhead
    $v = 15 = 1 + 4 + 8 + 2$ & $4^{1}\, 2^{5}\, -1^{4}\, -2^{5}$ & $\{4,2,1;1,1,4\}$ \\
$v = 35 = 1 + 6 + 24 + 4$ & $6^{1}\, 2.45^{14}\, -1^{6}\, -2.45^{14}$ & $\{6,4,1;1,1,6\}$ \\
$v = 42 = 1 + 6 + 30 + 5$ & $6^{1}\, 2^{21}\, -1^{6}\, -3^{14}$ & $\{6,5,1;1,1,6\}$ \\
$v = 24 = 1 + 7 + 14 + 2$ & $7^{1}\, 2.646^{8}\, -1^{7}\, -2.646^{8}$ & $\{7,4,1;1,2,7\}$ \\
$v = 45 = 1 + 8 + 32 + 4$ & $8^{1}\, 4^{12}\, -1^{8}\, -2^{24}$ & $\{8,4,1;1,1,8\}$ \\
$v = 63 = 1 + 8 + 48 + 6$ & $8^{1}\, 2.828^{27}\, -1^{8}\, -2.828^{27}$ & $\{8,6,1;1,1,8\}$ \\
$v = 27 = 1 + 8 + 16 + 2$ & $8^{1}\, 2^{12}\, -1^{8}\, -4^{6}$ & $\{8,6,1;1,3,8\}$ \\
$v = 40 = 1 + 9 + 27 + 3$ & $9^{1}\, 3^{15}\, -1^{9}\, -3^{15}$ & $\{9,6,1;1,2,9\}$ \\
$v = 33 = 1 + 10 + 20 + 2$ & $10^{1}\, 3.162^{11}\, -1^{10}\, -3.162^{11}$ & $\{10,6,1;1,3,10\}$ \\
$v = 99 = 1 + 10 + 80 + 8$ & $10^{1}\, 3.162^{44}\, -1^{10}\, -3.162^{44}$ & $\{10,8,1;1,1,10\}$ \\
$v = 60 = 1 + 11 + 44 + 4$ & $11^{1}\, 3.317^{24}\, -1^{11}\, -3.317^{24}$ & $\{11,8,1;1,2,11\}$ \\
$v = 143 = 1 + 12 + 120 + 10$ & $12^{1}\, 3.464^{65}\, -1^{12}\, -3.464^{65}$ & $\{12,10,1;1,1,12\}$ \\
$v = 42 = 1 + 13 + 26 + 2$ & $13^{1}\, 3.606^{14}\, -1^{13}\, -3.606^{14}$ & $\{13,8,1;1,4,13\}$ \\
$v = 84 = 1 + 13 + 65 + 5$ & $13^{1}\, 3.606^{35}\, -1^{13}\, -3.606^{35}$ & $\{13,10,1;1,2,13\}$ \\
$v = 195 = 1 + 14 + 168 + 12$ & $14^{1}\, 3.742^{90}\, -1^{14}\, -3.742^{90}$ & $\{14,12,1;1,1,14\}$ \\
$v = 48 = 1 + 15 + 30 + 2$ & $15^{1}\, 5^{12}\, -1^{15}\, -3^{20}$ & $\{15,8,1;1,4,15\}$ \\
$v = 96 = 1 + 15 + 75 + 5$ & $15^{1}\, 5^{30}\, -1^{15}\, -3^{50}$ & $\{15,10,1;1,2,15\}$ \\
$v = 112 = 1 + 15 + 90 + 6$ & $15^{1}\, 3.873^{48}\, -1^{15}\, -3.873^{48}$ & $\{15,12,1;1,2,15\}$ \\
$v = 64 = 1 + 15 + 45 + 3$ & $15^{1}\, 3^{30}\, -1^{15}\, -5^{18}$ & $\{15,12,1;1,4,15\}$ \\
$v = 128 = 1 + 15 + 105 + 7$ & $15^{1}\, 3^{70}\, -1^{15}\, -5^{42}$ & $\{15,14,1;1,2,15\}$ \\
$v = 51 = 1 + 16 + 32 + 2$ & $16^{1}\, 4^{17}\, -1^{16}\, -4^{17}$ & $\{16,10,1;1,5,16\}$ \\
$v = 85 = 1 + 16 + 64 + 4$ & $16^{1}\, 4^{34}\, -1^{16}\, -4^{34}$ & $\{16,12,1;1,3,16\}$ \\
$v = 255 = 1 + 16 + 224 + 14$ & $16^{1}\, 4^{119}\, -1^{16}\, -4^{119}$ & $\{16,14,1;1,1,16\}$ \\
$v = 72 = 1 + 17 + 51 + 3$ & $17^{1}\, 4.123^{27}\, -1^{17}\, -4.123^{27}$ & $\{17,12,1;1,4,17\}$ \\
$v = 144 = 1 + 17 + 119 + 7$ & $17^{1}\, 4.123^{63}\, -1^{17}\, -4.123^{63}$ & $\{17,14,1;1,2,17\}$ \\
$v = 133 = 1 + 18 + 108 + 6$ & $18^{1}\, 6^{38}\, -1^{18}\, -3^{76}$ & $\{18,12,1;1,2,18\}$ \\
$v = 76 = 1 + 18 + 54 + 3$ & $18^{1}\, 3^{38}\, -1^{18}\, -6^{19}$ & $\{18,15,1;1,5,18\}$ \\
$v = 323 = 1 + 18 + 288 + 16$ & $18^{1}\, 4.243^{152}\, -1^{18}\, -4.243^{152}$ & $\{18,16,1;1,1,18\}$ \\
$v = 60 = 1 + 19 + 38 + 2$ & $19^{1}\, 4.359^{20}\, -1^{19}\, -4.359^{20}$ & $\{19,12,1;1,6,19\}$ \\
$v = 180 = 1 + 19 + 152 + 8$ & $19^{1}\, 4.359^{80}\, -1^{19}\, -4.359^{80}$ & $\{19,16,1;1,2,19\}$ \\
$v = 231 = 1 + 20 + 200 + 10$ & $20^{1}\, 10^{35}\, -1^{20}\, -2^{175}$ & $\{20,10,1;1,1,20\}$ \\
$v = 84 = 1 + 20 + 60 + 3$ & $20^{1}\, 4^{35}\, -1^{20}\, -5^{28}$ & $\{20,15,1;1,5,20\}$ \\
$v = 399 = 1 + 20 + 360 + 18$ & $20^{1}\, 4.472^{189}\, -1^{20}\, -4.472^{189}$ & $\{20,18,1;1,1,20\}$ \\
$v = 210 = 1 + 20 + 180 + 9$ & $20^{1}\, 4^{105}\, -1^{20}\, -5^{84}$ & $\{20,18,1;1,2,20\}$ \\
$v = 110 = 1 + 21 + 84 + 4$ & $21^{1}\, 4.583^{44}\, -1^{21}\, -4.583^{44}$ & $\{21,16,1;1,4,21\}$ \\
$v = 220 = 1 + 21 + 189 + 9$ & $21^{1}\, 4.583^{99}\, -1^{21}\, -4.583^{99}$ & $\{21,18,1;1,2,21\}$ \\
$v = 132 = 1 + 21 + 105 + 5$ & $21^{1}\, 3^{77}\, -1^{21}\, -7^{33}$ & $\{21,20,1;1,4,21\}$ \\
$v = 69 = 1 + 22 + 44 + 2$ & $22^{1}\, 4.69^{23}\, -1^{22}\, -4.69^{23}$ & $\{22,14,1;1,7,22\}$ \\
$v = 161 = 1 + 22 + 132 + 6$ & $22^{1}\, 4.69^{69}\, -1^{22}\, -4.69^{69}$ & $\{22,18,1;1,3,22\}$ \\
$v = 483 = 1 + 22 + 440 + 20$ & $22^{1}\, 4.69^{230}\, -1^{22}\, -4.69^{230}$ & $\{22,20,1;1,1,22\}$ \\
$v = 264 = 1 + 23 + 230 + 10$ & $23^{1}\, 4.796^{120}\, -1^{23}\, -4.796^{120}$ & $\{23,20,1;1,2,23\}$ \\
$v = 75 = 1 + 24 + 48 + 2$ & $24^{1}\, 6^{20}\, -1^{24}\, -4^{30}$ & $\{24,14,1;1,7,24\}$ \\
$v = 175 = 1 + 24 + 144 + 6$ & $24^{1}\, 6^{60}\, -1^{24}\, -4^{90}$ & $\{24,18,1;1,3,24\}$ \\
$v = 525 = 1 + 24 + 480 + 20$ & $24^{1}\, 6^{200}\, -1^{24}\, -4^{300}$ & $\{24,20,1;1,1,24\}$ \\
$v = 125 = 1 + 24 + 96 + 4$ & $24^{1}\, 4^{60}\, -1^{24}\, -6^{40}$ & $\{24,20,1;1,5,24\}$ \\
$v = 575 = 1 + 24 + 528 + 22$ & $24^{1}\, 4.899^{275}\, -1^{24}\, -4.899^{275}$ & $\{24,22,1;1,1,24\}$ \\
$v = 78 = 1 + 25 + 50 + 2$ & $25^{1}\, 5^{26}\, -1^{25}\, -5^{26}$ & $\{25,16,1;1,8,25\}$ \\
$v = 104 = 1 + 25 + 75 + 3$ & $25^{1}\, 5^{39}\, -1^{25}\, -5^{39}$ & $\{25,18,1;1,6,25\}$ \\
$v = 156 = 1 + 25 + 125 + 5$ & $25^{1}\, 5^{65}\, -1^{25}\, -5^{65}$ & $\{25,20,1;1,4,25\}$ \\
$v = 312 = 1 + 25 + 275 + 11$ & $25^{1}\, 5^{143}\, -1^{25}\, -5^{143}$ & $\{25,22,1;1,2,25\}$ \\
$v = 135 = 1 + 26 + 104 + 4$ & $26^{1}\, 5.099^{54}\, -1^{26}\, -5.099^{54}$ & $\{26,20,1;1,5,26\}$ \\
$v = 675 = 1 + 26 + 624 + 24$ & $26^{1}\, 5.099^{324}\, -1^{26}\, -5.099^{324}$ & $\{26,24,1;1,1,26\}$ \\
$v = 140 = 1 + 27 + 108 + 4$ & $27^{1}\, 9^{28}\, -1^{27}\, -3^{84}$ & $\{27,16,1;1,4,27\}$ \\
$v = 280 = 1 + 27 + 243 + 9$ & $27^{1}\, 9^{63}\, -1^{27}\, -3^{189}$ & $\{27,18,1;1,2,27\}$ \\
$v = 364 = 1 + 27 + 324 + 12$ & $27^{1}\, 5.196^{168}\, -1^{27}\, -5.196^{168}$ & $\{27,24,1;1,2,27\}$ \\
$v = 112 = 1 + 27 + 81 + 3$ & $27^{1}\, 3^{63}\, -1^{27}\, -9^{21}$ & $\{27,24,1;1,8,27\}$ \\
$v = 87 = 1 + 28 + 56 + 2$ & $28^{1}\, 5.292^{29}\, -1^{28}\, -5.292^{29}$ & $\{28,18,1;1,9,28\}$ \\
$v = 348 = 1 + 28 + 308 + 11$ & $28^{1}\, 7^{116}\, -1^{28}\, -4^{203}$ & $\{28,22,1;1,2,28\}$ \\
$v = 261 = 1 + 28 + 224 + 8$ & $28^{1}\, 5.292^{116}\, -1^{28}\, -5.292^{116}$ & $\{28,24,1;1,3,28\}$ \\
$v = 783 = 1 + 28 + 728 + 26$ & $28^{1}\, 5.292^{377}\, -1^{28}\, -5.292^{377}$ & $\{28,26,1;1,1,28\}$ \\
$v = 210 = 1 + 29 + 174 + 6$ & $29^{1}\, 5.385^{90}\, -1^{29}\, -5.385^{90}$ & $\{29,24,1;1,4,29\}$ \\
$v = 420 = 1 + 29 + 377 + 13$ & $29^{1}\, 5.385^{195}\, -1^{29}\, -5.385^{195}$ & $\{29,26,1;1,2,29\}$ \\
$v = 899 = 1 + 30 + 840 + 28$ & $30^{1}\, 5.477^{434}\, -1^{30}\, -5.477^{434}$ & $\{30,28,1;1,1,30\}$ \\
$v = 96 = 1 + 31 + 62 + 2$ & $31^{1}\, 5.568^{32}\, -1^{31}\, -5.568^{32}$ & $\{31,20,1;1,10,31\}$ \\
$v = 160 = 1 + 31 + 124 + 4$ & $31^{1}\, 5.568^{64}\, -1^{31}\, -5.568^{64}$ & $\{31,24,1;1,6,31\}$ \\
$v = 480 = 1 + 31 + 434 + 14$ & $31^{1}\, 5.568^{224}\, -1^{31}\, -5.568^{224}$ & $\{31,28,1;1,2,31\}$ \\
$v = 99 = 1 + 32 + 64 + 2$ & $32^{1}\, 8^{22}\, -1^{32}\, -4^{44}$ & $\{32,18,1;1,9,32\}$ \\
$v = 297 = 1 + 32 + 256 + 8$ & $32^{1}\, 8^{88}\, -1^{32}\, -4^{176}$ & $\{32,24,1;1,3,32\}$ \\
$v = 891 = 1 + 32 + 832 + 26$ & $32^{1}\, 8^{286}\, -1^{32}\, -4^{572}$ & $\{32,26,1;1,1,32\}$ \\
$v = 165 = 1 + 32 + 128 + 4$ & $32^{1}\, 4^{88}\, -1^{32}\, -8^{44}$ & $\{32,28,1;1,7,32\}$ \\
$v = 1023 = 1 + 32 + 960 + 30$ & $32^{1}\, 5.657^{495}\, -1^{32}\, -5.657^{495}$ & $\{32,30,1;1,1,32\}$ \\
$v = 231 = 1 + 32 + 192 + 6$ & $32^{1}\, 4^{132}\, -1^{32}\, -8^{66}$ & $\{32,30,1;1,5,32\}$ \\
$v = 136 = 1 + 33 + 99 + 3$ & $33^{1}\, 5.745^{51}\, -1^{33}\, -5.745^{51}$ & $\{33,24,1;1,8,33\}$ \\
$v = 272 = 1 + 33 + 231 + 7$ & $33^{1}\, 5.745^{119}\, -1^{33}\, -5.745^{119}$ & $\{33,28,1;1,4,33\}$ \\
$v = 544 = 1 + 33 + 495 + 15$ & $33^{1}\, 5.745^{255}\, -1^{33}\, -5.745^{255}$ & $\{33,30,1;1,2,33\}$ \\
$v = 105 = 1 + 34 + 68 + 2$ & $34^{1}\, 5.831^{35}\, -1^{34}\, -5.831^{35}$ & $\{34,22,1;1,11,34\}$ \\
$v = 385 = 1 + 34 + 340 + 10$ & $34^{1}\, 5.831^{175}\, -1^{34}\, -5.831^{175}$ & $\{34,30,1;1,3,34\}$ \\
$v = 1155 = 1 + 34 + 1088 + 32$ & $34^{1}\, 5.831^{560}\, -1^{34}\, -5.831^{560}$ & $\{34,32,1;1,1,34\}$ \\
$v = 144 = 1 + 35 + 105 + 3$ & $35^{1}\, 7^{45}\, -1^{35}\, -5^{63}$ & $\{35,24,1;1,8,35\}$ \\
$v = 108 = 1 + 35 + 70 + 2$ & $35^{1}\, 5^{42}\, -1^{35}\, -7^{30}$ & $\{35,24,1;1,12,35\}$ \\
$v = 288 = 1 + 35 + 245 + 7$ & $35^{1}\, 7^{105}\, -1^{35}\, -5^{147}$ & $\{35,28,1;1,4,35\}$ \\
$v = 576 = 1 + 35 + 525 + 15$ & $35^{1}\, 7^{225}\, -1^{35}\, -5^{315}$ & $\{35,30,1;1,2,35\}$ \\
$v = 216 = 1 + 35 + 175 + 5$ & $35^{1}\, 5^{105}\, -1^{35}\, -7^{75}$ & $\{35,30,1;1,6,35\}$ \\
$v = 612 = 1 + 35 + 560 + 16$ & $35^{1}\, 5.916^{288}\, -1^{35}\, -5.916^{288}$ & $\{35,32,1;1,2,35\}$ \\
$v = 324 = 1 + 35 + 280 + 8$ & $35^{1}\, 5^{168}\, -1^{35}\, -7^{120}$ & $\{35,32,1;1,4,35\}$ \\
$v = 648 = 1 + 35 + 595 + 17$ & $35^{1}\, 5^{357}\, -1^{35}\, -7^{255}$ & $\{35,34,1;1,2,35\}$ \\
$v = 962 = 1 + 36 + 900 + 25$ & $36^{1}\, 12^{185}\, -1^{36}\, -3^{740}$ & $\{36,25,1;1,1,36\}$ \\
$v = 185 = 1 + 36 + 144 + 4$ & $36^{1}\, 6^{74}\, -1^{36}\, -6^{74}$ & $\{36,28,1;1,7,36\}$ \\
$v = 259 = 1 + 36 + 216 + 6$ & $36^{1}\, 6^{111}\, -1^{36}\, -6^{111}$ & $\{36,30,1;1,5,36\}$ \\
$v = 1295 = 1 + 36 + 1224 + 34$ & $36^{1}\, 6^{629}\, -1^{36}\, -6^{629}$ & $\{36,34,1;1,1,36\}$ \\
$v = 114 = 1 + 37 + 74 + 2$ & $37^{1}\, 6.083^{38}\, -1^{37}\, -6.083^{38}$ & $\{37,24,1;1,12,37\}$ \\
$v = 228 = 1 + 37 + 185 + 5$ & $37^{1}\, 6.083^{95}\, -1^{37}\, -6.083^{95}$ & $\{37,30,1;1,6,37\}$ \\
$v = 342 = 1 + 37 + 296 + 8$ & $37^{1}\, 6.083^{152}\, -1^{37}\, -6.083^{152}$ & $\{37,32,1;1,4,37\}$ \\
$v = 684 = 1 + 37 + 629 + 17$ & $37^{1}\, 6.083^{323}\, -1^{37}\, -6.083^{323}$ & $\{37,34,1;1,2,37\}$ \\
$v = 1443 = 1 + 38 + 1368 + 36$ & $38^{1}\, 6.164^{702}\, -1^{38}\, -6.164^{702}$ & $\{38,36,1;1,1,38\}$ \\
$v = 280 = 1 + 39 + 234 + 6$ & $39^{1}\, 13^{45}\, -1^{39}\, -3^{195}$ & $\{39,24,1;1,4,39\}$ \\
$v = 760 = 1 + 39 + 702 + 18$ & $39^{1}\, 6.245^{360}\, -1^{39}\, -6.245^{360}$ & $\{39,36,1;1,2,39\}$ \\
$v = 123 = 1 + 40 + 80 + 2$ & $40^{1}\, 6.325^{41}\, -1^{40}\, -6.325^{41}$ & $\{40,26,1;1,13,40\}$ \\
$v = 533 = 1 + 40 + 480 + 12$ & $40^{1}\, 6.325^{246}\, -1^{40}\, -6.325^{246}$ & $\{40,36,1;1,3,40\}$ \\
$v = 1599 = 1 + 40 + 1520 + 38$ & $40^{1}\, 6.325^{779}\, -1^{40}\, -6.325^{779}$ & $\{40,38,1;1,1,40\}$ \\
$v = 574 = 1 + 40 + 520 + 13$ & $40^{1}\, 5^{328}\, -1^{40}\, -8^{205}$ & $\{40,39,1;1,3,40\}$ \\
$v = 168 = 1 + 41 + 123 + 3$ & $41^{1}\, 6.403^{63}\, -1^{41}\, -6.403^{63}$ & $\{41,30,1;1,10,41\}$ \\
$v = 210 = 1 + 41 + 164 + 4$ & $41^{1}\, 6.403^{84}\, -1^{41}\, -6.403^{84}$ & $\{41,32,1;1,8,41\}$ \\
$v = 420 = 1 + 41 + 369 + 9$ & $41^{1}\, 6.403^{189}\, -1^{41}\, -6.403^{189}$ & $\{41,36,1;1,4,41\}$ \\
$v = 840 = 1 + 41 + 779 + 19$ & $41^{1}\, 6.403^{399}\, -1^{41}\, -6.403^{399}$ & $\{41,38,1;1,2,41\}$ \\
$v = 1720 = 1 + 42 + 1638 + 39$ & $42^{1}\, 7^{774}\, -1^{42}\, -6^{903}$ & $\{42,39,1;1,1,42\}$ \\
$v = 602 = 1 + 42 + 546 + 13$ & $42^{1}\, 6^{301}\, -1^{42}\, -7^{258}$ & $\{42,39,1;1,3,42\}$ \\
$v = 1763 = 1 + 42 + 1680 + 40$ & $42^{1}\, 6.481^{860}\, -1^{42}\, -6.481^{860}$ & $\{42,40,1;1,1,42\}$ \\
$v = 132 = 1 + 43 + 86 + 2$ & $43^{1}\, 6.557^{44}\, -1^{43}\, -6.557^{44}$ & $\{43,28,1;1,14,43\}$ \\
$v = 308 = 1 + 43 + 258 + 6$ & $43^{1}\, 6.557^{132}\, -1^{43}\, -6.557^{132}$ & $\{43,36,1;1,6,43\}$ \\
$v = 924 = 1 + 43 + 860 + 20$ & $43^{1}\, 6.557^{440}\, -1^{43}\, -6.557^{440}$ & $\{43,40,1;1,2,43\}$ \\
$v = 135 = 1 + 44 + 88 + 2$ & $44^{1}\, 11^{24}\, -1^{44}\, -4^{66}$ & $\{44,24,1;1,12,44\}$ \\
$v = 180 = 1 + 44 + 132 + 3$ & $44^{1}\, 11^{36}\, -1^{44}\, -4^{99}$ & $\{44,27,1;1,9,44\}$ \\
$v = 270 = 1 + 44 + 220 + 5$ & $44^{1}\, 11^{60}\, -1^{44}\, -4^{165}$ & $\{44,30,1;1,6,44\}$ \\
$v = 405 = 1 + 44 + 352 + 8$ & $44^{1}\, 11^{96}\, -1^{44}\, -4^{264}$ & $\{44,32,1;1,4,44\}$ \\
$v = 540 = 1 + 44 + 484 + 11$ & $44^{1}\, 11^{132}\, -1^{44}\, -4^{363}$ & $\{44,33,1;1,3,44\}$ \\
$v = 810 = 1 + 44 + 748 + 17$ & $44^{1}\, 11^{204}\, -1^{44}\, -4^{561}$ & $\{44,34,1;1,2,44\}$ \\
$v = 1620 = 1 + 44 + 1540 + 35$ & $44^{1}\, 11^{420}\, -1^{44}\, -4^{1155}$ & $\{44,35,1;1,1,44\}$ \\
$v = 225 = 1 + 44 + 176 + 4$ & $44^{1}\, 4^{132}\, -1^{44}\, -11^{48}$ & $\{44,40,1;1,10,44\}$ \\
$v = 1935 = 1 + 44 + 1848 + 42$ & $44^{1}\, 6.633^{945}\, -1^{44}\, -6.633^{945}$ & $\{44,42,1;1,1,44\}$ \\
$v = 184 = 1 + 45 + 135 + 3$ & $45^{1}\, 15^{23}\, -1^{45}\, -3^{115}$ & $\{45,24,1;1,8,45\}$ \\
$v = 736 = 1 + 45 + 675 + 15$ & $45^{1}\, 15^{115}\, -1^{45}\, -3^{575}$ & $\{45,30,1;1,2,45\}$ \\
$v = 506 = 1 + 45 + 450 + 10$ & $45^{1}\, 6.708^{230}\, -1^{45}\, -6.708^{230}$ & $\{45,40,1;1,4,45\}$ \\
$v = 1012 = 1 + 45 + 945 + 21$ & $45^{1}\, 6.708^{483}\, -1^{45}\, -6.708^{483}$ & $\{45,42,1;1,2,45\}$ \\
$v = 368 = 1 + 45 + 315 + 7$ & $45^{1}\, 5^{207}\, -1^{45}\, -9^{115}$ & $\{45,42,1;1,6,45\}$ \\
$v = 141 = 1 + 46 + 92 + 2$ & $46^{1}\, 6.782^{47}\, -1^{46}\, -6.782^{47}$ & $\{46,30,1;1,15,46\}$ \\
$v = 235 = 1 + 46 + 184 + 4$ & $46^{1}\, 6.782^{94}\, -1^{46}\, -6.782^{94}$ & $\{46,36,1;1,9,46\}$ \\
$v = 423 = 1 + 46 + 368 + 8$ & $46^{1}\, 6.782^{188}\, -1^{46}\, -6.782^{188}$ & $\{46,40,1;1,5,46\}$ \\
$v = 705 = 1 + 46 + 644 + 14$ & $46^{1}\, 6.782^{329}\, -1^{46}\, -6.782^{329}$ & $\{46,42,1;1,3,46\}$ \\
$v = 2115 = 1 + 46 + 2024 + 44$ & $46^{1}\, 6.782^{1034}\, -1^{46}\, -6.782^{1034}$ & $\{46,44,1;1,1,46\}$ \\
$v = 1104 = 1 + 47 + 1034 + 22$ & $47^{1}\, 6.856^{528}\, -1^{47}\, -6.856^{528}$ & $\{47,44,1;1,2,47\}$ \\
$v = 147 = 1 + 48 + 96 + 2$ & $48^{1}\, 8^{42}\, -1^{48}\, -6^{56}$ & $\{48,30,1;1,15,48\}$ \\
$v = 637 = 1 + 48 + 576 + 12$ & $48^{1}\, 12^{147}\, -1^{48}\, -4^{441}$ & $\{48,36,1;1,3,48\}$ \\
$v = 245 = 1 + 48 + 192 + 4$ & $48^{1}\, 8^{84}\, -1^{48}\, -6^{112}$ & $\{48,36,1;1,9,48\}$ \\
$v = 441 = 1 + 48 + 384 + 8$ & $48^{1}\, 8^{168}\, -1^{48}\, -6^{224}$ & $\{48,40,1;1,5,48\}$ \\
$v = 735 = 1 + 48 + 672 + 14$ & $48^{1}\, 8^{294}\, -1^{48}\, -6^{392}$ & $\{48,42,1;1,3,48\}$ \\
$v = 343 = 1 + 48 + 288 + 6$ & $48^{1}\, 6^{168}\, -1^{48}\, -8^{126}$ & $\{48,42,1;1,7,48\}$ \\
$v = 2205 = 1 + 48 + 2112 + 44$ & $48^{1}\, 8^{924}\, -1^{48}\, -6^{1232}$ & $\{48,44,1;1,1,48\}$ \\
$v = 245 = 1 + 48 + 192 + 4$ & $48^{1}\, 4^{147}\, -1^{48}\, -12^{49}$ & $\{48,44,1;1,11,48\}$ \\
$v = 2303 = 1 + 48 + 2208 + 46$ & $48^{1}\, 6.928^{1127}\, -1^{48}\, -6.928^{1127}$ & $\{48,46,1;1,1,48\}$ \\
$v = 150 = 1 + 49 + 98 + 2$ & $49^{1}\, 7^{50}\, -1^{49}\, -7^{50}$ & $\{49,32,1;1,16,49\}$ \\
$v = 200 = 1 + 49 + 147 + 3$ & $49^{1}\, 7^{75}\, -1^{49}\, -7^{75}$ & $\{49,36,1;1,12,49\}$ \\
$v = 300 = 1 + 49 + 245 + 5$ & $49^{1}\, 7^{125}\, -1^{49}\, -7^{125}$ & $\{49,40,1;1,8,49\}$ \\
$v = 400 = 1 + 49 + 343 + 7$ & $49^{1}\, 7^{175}\, -1^{49}\, -7^{175}$ & $\{49,42,1;1,6,49\}$ \\
$v = 600 = 1 + 49 + 539 + 11$ & $49^{1}\, 7^{275}\, -1^{49}\, -7^{275}$ & $\{49,44,1;1,4,49\}$ \\
$v = 1200 = 1 + 49 + 1127 + 23$ & $49^{1}\, 7^{575}\, -1^{49}\, -7^{575}$ & $\{49,46,1;1,2,49\}$ \\
$v = 204 = 1 + 50 + 150 + 3$ & $50^{1}\, 10^{51}\, -1^{50}\, -5^{102}$ & $\{50,33,1;1,11,50\}$ \\
$v = 153 = 1 + 50 + 100 + 2$ & $50^{1}\, 5^{68}\, -1^{50}\, -10^{34}$ & $\{50,36,1;1,18,50\}$ \\
$v = 561 = 1 + 50 + 500 + 10$ & $50^{1}\, 10^{170}\, -1^{50}\, -5^{340}$ & $\{50,40,1;1,4,50\}$ \\
$v = 1122 = 1 + 50 + 1050 + 21$ & $50^{1}\, 10^{357}\, -1^{50}\, -5^{714}$ & $\{50,42,1;1,2,50\}$ \\
$v = 357 = 1 + 50 + 300 + 6$ & $50^{1}\, 7.071^{153}\, -1^{50}\, -7.071^{153}$ & $\{50,42,1;1,7,50\}$ \\
$v = 2244 = 1 + 50 + 2150 + 43$ & $50^{1}\, 10^{731}\, -1^{50}\, -5^{1462}$ & $\{50,43,1;1,1,50\}$ \\
$v = 306 = 1 + 50 + 250 + 5$ & $50^{1}\, 5^{170}\, -1^{50}\, -10^{85}$ & $\{50,45,1;1,9,50\}$ \\
$v = 2499 = 1 + 50 + 2400 + 48$ & $50^{1}\, 7.071^{1224}\, -1^{50}\, -7.071^{1224}$ & $\{50,48,1;1,1,50\}$ \\
$v = 459 = 1 + 50 + 400 + 8$ & $50^{1}\, 5^{272}\, -1^{50}\, -10^{136}$ & $\{50,48,1;1,6,50\}$ \\

    \end{longtable} }

\end{document}